\documentclass[reqno,12pt]{amsart}
\usepackage[dvips]{color}
\usepackage{amsmath, latexsym, amsfonts, amssymb, amsthm, amscd,mathrsfs,bbm}
\usepackage{mathtools}
\definecolor{myblue}{rgb}{0.09,0.32,0.44} %22-84-113
\usepackage[pdfborder={0 0 0},pdfborderstyle={},colorlinks,linkcolor=myblue,citecolor=myblue,urlcolor=blue]{hyperref}
\usepackage{graphics,epsf,psfrag}
\numberwithin{equation}{section}

\newtheorem{maintheorem}{Theorem}
\newtheorem{theorem}{Theorem}[section]

\newtheorem{lemma}[theorem]{Lemma}
\newtheorem{proposition}[theorem]{Proposition}

\newtheorem{definition}[theorem]{Definition}

\newtheorem{question}[theorem]{Question}

\theoremstyle{remark}
\newtheorem{remark}[theorem]{Remark}
\newtheorem*{remarks*}{Remarks}

\newcommand{\R}{\mathbb{R}}

\newcommand{\N}{\mathbb{N}}

\newcommand{\cF}{{\ensuremath{\mathcal F}} }

\newcommand{\cH}{{\ensuremath{\mathcal H}} }

\newcommand{\cT}{{\ensuremath{\mathcal T}} }

\newcommand{\cJ}{{\ensuremath{\mathcal J}} }
\newcommand{\cD}{{\ensuremath{\mathcal D}} }

\newcommand{\cU}{{\ensuremath{\mathcal U}} }

\DeclareMathSymbol{\leqslant}{\mathalpha}{AMSa}{"36} % nicer `smaller or equal'
\DeclareMathSymbol{\geqslant}{\mathalpha}{AMSa}{"3E} % nicer `larger or equal'
\DeclareMathSymbol{\eset}{\mathalpha}{AMSb}{"3F}     % nicer `emptyset'
\newcommand{\bbN}{{\ensuremath{\mathbb N}} }

\newcommand{\gl}{\lambda}

\renewcommand{\epsilon}{\varepsilon}

\makeatletter
\def\captionfont@{\footnotesize}
\def\captionheadfont@{\scshape}

\long\def\@makecaption#1#2{%
  \vspace{2mm}
  \setbox\@tempboxa\vbox{\color@setgroup
    \advance\hsize-6pc\noindent
    \captionfont@\captionheadfont@#1\@xp\@ifnotempty\@xp
        {\@cdr#2\@nil}{.\captionfont@\upshape\enspace#2}%
    \unskip\kern-6pc\par
    \global\setbox\@ne\lastbox\color@endgroup}%
  \ifhbox\@ne % the normal case
    \setbox\@ne\hbox{\unhbox\@ne\unskip\unskip\unpenalty\unkern}%
  \fi
  \ifdim\wd\@tempboxa=\z@ % this means caption will fit on one line
    \setbox\@ne\hbox to\columnwidth{\hss\kern-6pc\box\@ne\hss}%
  \else % tempboxa contained more than one line
    \setbox\@ne\vbox{\unvbox\@tempboxa\parskip\z@skip
        \noindent\unhbox\@ne\advance\hsize-6pc\par}%
\fi
  \ifnum\@tempcnta<64 % if the float IS a figure...
    \addvspace\abovecaptionskip
    \moveright 3pc\box\@ne
  \else % if the float IS NOT a figure...
    \moveright 3pc\box\@ne
    \nobreak
    \vskip\belowcaptionskip
  \fi
\relax
}
\makeatother
\def\writefig#1 #2 #3 {\rlap{\kern #1 truecm
\raise #2 truecm \hbox{#3}}}

\renewcommand{\Pr}{ \mathbb P}
\newcommand{\ex}{\mathbb E}
\newcommand{ \rel}{ t_{\mathrm{rel}} }

\newcommand{ \mix}{ t_{\mathrm{mix}} }

\newcommand{ \TV}{ \mathrm{TV} }

\newcommand{ \cM}{ \mathcal M }

\newcommand{ \cL}{ \mathcal L }

\newcommand{\eps}{\epsilon }

\usepackage[normalem]{ulem}
\usepackage{color}

\newcommand{\tb}[1]{\textbf{\color{blue}{#1}}}

\makeatletter
\@namedef{subjclassname@2020}{%
  \textup{2020} Mathematics Subject Classification}
\makeatother

\newcommand{\sym}{\mathfrak{S}}
\DeclareMathOperator{\Cay}{Cay}

\makeatletter
\def\moverlay{\mathpalette\mov@rlay}
\def\mov@rlay#1#2{\leavevmode\vtop{%
   \baselineskip\z@skip \lineskiplimit-\maxdimen
   \ialign{\hfil$\m@th#1##$\hfil\cr#2\crcr}}}
\newcommand{\charfusion}[3][\mathord]{
    #1{\ifx#1\mathop\vphantom{#2}\fi
        \mathpalette\mov@rlay{#2\cr#3}
      }
    \ifx#1\mathop\expandafter\displaylimits\fi}
\makeatother

\newcommand{\sdfrac}[2]{\mbox{\small$\displaystyle\frac{#1}{#2}$}}

\allowdisplaybreaks

\begin{document}

\title{Sensitivity of mixing times of Cayley graphs}
\author{Jonathan Hermon \and Gady Kozma}
\thanks{
University of British Columbia, Vancouver, Canada. E-mail: {\tt jhermon@math.ubc.ca}.
Supported by NSERC grant}
\thanks{
Weizmann Institute of Science, Rehovot, Israel. E-mail: {\tt Gady.Kozma@weizmann.ac.il}.}

\date{}

\begin{abstract}
  We show that the total variation mixing time is not quasi-isometry invariant, even for Cayley graphs. Namely, we construct a sequence of pairs of Cayley graphs with maps between them that twist the metric in a bounded way, while the ratio of the two mixing times goes to infinity. The Cayley graphs serving as an example have unbounded degrees.

  For non-transitive graphs we construct \emph{bounded degree} graphs for which the mixing time from the worst starting point for one graph is asymptotically smaller than the mixing time from the best starting point of the random walk on a network obtained by increasing some of the edge weights from 1 to $1+o(1)$.

\end{abstract}



\maketitle

\section{Introduction}

There are numerous works aiming at sharp geometric bounds on the mixing time of a finite Markov chain. Examples include Morris and Peres' evolving sets bound \cite{cf:Evolving}, expressed in terms of the expansion profile, and the related bound by Fountoulakis and Reed \cite{FR}. The sharpest geometric bounds on the uniform  (a.k.a.\ $L_{\infty}$) mixing time are given in terms of the Log-Sobolev constant (see \cite{diaconis} for a survey on the topic) and the spectral profile bound, due to Goel et al.~\cite{cf:Spectral}. Both determine the uniform  mixing time up to a multiplicative factor of order $\log \log [ 1/\min \pi (x)] $, where throughout $\pi$ denotes the stationary distribution (see \cite{diaconis,cf:Kozma}). The reader not familiar with  mixing time definitions can find them in \S\ref{s:mixingpre}. Other notions and definitions used below can be found in \S\S\ref{s:qirob} and \ref{s:notation}.

This type of geometric bounds on mixing times are robust under bounded perturbations of the edge weights, and in the bounded degree setup, also under quasi-isometries. % (see \S\S\ref{s:qirob} and \ref{s:com}-\ref{s:mixingpre} for a detailed discussion).
That is, changing some of the edge weights by at most some multiplicative constant factor can change these geometric bounds only by some corresponding constant factor. % More generally, we call a quantity associated with Markov chains  \emph{robust} if it can vary by at most a factor $\delta(c)$ (for some $\delta : (0,\infty) \to (0,\infty)$) between any two ergodic Markov chains on the same finite state space whose stationary distributions (pointwise) and Dirichlet forms vary by at most a factor $c$ (see \eqref{e:DFcomparison} for a more precise formulation).\footnote{One can also consider the case that the state spaces vary, but we do not do so here.}  
A natural question, with obvious implications to the potential sharpness of such geometric bounds, is whether mixing times are themselves robust under small changes to the geometry of the Markov chain. For instance, can bounded perturbations of the edge weights change the mixing time by more than a constant factor? Similarly,  how far apart can the mixing times of  simple random walks (SRW) on two quasi-isometric graphs of bounded degree be?
% A related question is whether mixing times can be characterized up to universal %constants (perhaps only under reversibility) using geometric quantities %or extremal characterizations which are robust.\footnote{That is, using %quantities which can change by at most some bounded factor under a bounded %perturbation of the edge weights, or for lazy simple random walk on a bounded %degree graph, under  rough-isometries.} 
Different variants of this question were asked by various authors such as Pittet and Saloff-Coste \cite[\S6]{cf:Pittet}, Diaconis and Saloff-Coste \cite[p.\ 720]{diaconis} and Aldous and Fill \cite[Open Problem 8.23]{cf:Aldous}.%\footnote{The last two references ask for an extremal characterization of the $L_{\infty}$-mixing time in terms of the Dirichlet form. Such characterization must be robust.}. 

Ding and Peres \cite{cf:Ding} constructed a sequence of bounded degree graphs satisfying that the order of the total variation mixing times strictly increases as a result of a certain sequence of bounded perturbations of the edge weights.\footnote{Their construction was refined by J.H.\ and Peres in \cite[Theorem 3]{HP}, so that the mixing time changes by an order of $\log |V|$, which is optimal. The same paper contains various additional results concerning sensitivity of mixing times and of the cutoff phenomenon under small changes to the geometry of the chain.} In \cite{unifsensitivity} a similar example is constructed in which the uniform mixing time is sensitive under bounded perturbations of the edge weights, as well as under a quasi-isometry. All these examples are based on the `perturbed tree' example of T.\ Lyons \cite{L87} (simplified by Benjamini \cite{cf:Benjamini}). In particular they are highly non-transitive, and a priori it appears as if what makes such examples work could not be imitated by a transitive example. % (see the discussion in \S\ref{s:IPheuristic} for more details).
It remained an open problem to determine whether the total variation mixing time of random walk on vertex-transitive graphs is robust under small perturbations. This was asked by Ding and Peres \cite[Question 1.4]{cf:Ding} (see also \cite[p.\ 3]{cf:Kozma} and  \cite[\S6]{cf:Pittet}). In this paper we give a negative answer to this question, even when the small perturbation  preserves transitivity. 

%\textcolor{blue}{We denote the group of permutations of a set of $n$ elements (the symmetric group) by $\sym_n$. Throughout, given a finite graph $G=(V,E)$ we identify $\sym_{|V|}$ with the group of permutations of the vertex set $V$. Loosely speaking, we say that a graph $G=(V,E)$ is $(a,b)$-quasi-isometric to a graph $G'=(V',E')$ (where $a,b \in \mathbb{N}$) if there is a map $\phi$ from $V$ to $V'$ whose image is an $a+b$ net for the graph $G'$ (i.e.\ every vertex of $G'$ is within graph distance at most $a+b$ from $\phi(V)$), which distorts distances by at most a multiplicative factor $a$ and an additive factor of at most $b$. The precise definition of $(a,b)$-quasi-isometry can be found in \S\ref{s:qirob} below. As usual, $x\gtrsim y$ (as well as $y \lesssim x$) means that there exists some universal constant $c>0$ such that $x\ge cy$ and   $x\asymp y$ means that  $x\gtrsim y$ and  $x\lesssim y$.}
We denote the group of permutations of $n$ elements by $\sym_n$. Recall that a transposition is an element of $\sym_n$ which exchanges two values and keeps all the rest fixed.
\begin{maintheorem}
\label{thm:short}
%\textcolor{blue}{There exist an absolute constant $N_0 \in \N$ and a pair of sequences of sets of transpositions $(S_n)_{n =N_{0}}^{\infty}$ and $(S_n')_{n =N_0}^{\infty}$ such that for all of $n \ge N_0$ the following hold: $S_n \subset S_n' \subset S_n^3:=\{xyz:x,y,z \in S_n \} \subset \sym_{n}$,  the Cayley graphs $\Cay(\sym_{n},S_n)$ and $\Cay(\sym_{n},S_n')$ are $(3,0)$-quasi-isometric}
There exist a pair of sequences of sets of transpositions $S_n$ and $S_n'$ such that the Cayley graphs $\Cay(\sym_{n},S_n)$ and $\Cay(\sym_{n},S_n')$ are $(3,0)$-quasi-isometric and
\[
\mix(\Cay(\sym_{n},S_n')) \gtrsim   \mix(\Cay(\sym_{n},S_{n})) \log \log \log |\sym_{n}|.
\]
Further, $S_n \subset S_n' \subset S_n^3:=\{xyz:x,y,z \in S_n \}$.
\end{maintheorem}
Of course, $\log\log\log|\sym_{n}|\asymp\log\log n$. We formulated the theorem in this way because the size of the group is the more natural object in this context. Let us remark that probably the ratio of mixing time in our example is indeed $\asymp \log\log\log |\sym_{n}|$, but for brevity we prove only the lower bound.

The mixing times in Theorem \ref{thm:short} are the total variation ones. In what comes, whenever we write mixing time without mentioning the metric, it is always the total variation mixing time.
The behaviour described in Theorem \ref{thm:short} cannot occur for the uniform mixing times which  in the transitive setup is quasi-isometry invariant (see Theorem \ref{thm:L2mix} below).

% As Theorem \ref{thm:short} is our main result, let us pause the general discussion to explain the structure of the sets $S_n$ and give some indications of the proof. %For brevity, we skip the definitions of some standard notions in \S\ref{s:sketch} --- whatever is missing should be found in \S \ref{s: Pre}.
% We return to the general discussion and to additional results (Theorems \ref{thm:weighted} and \ref{thm: 2}) in \S \ref{s:qirob}. 

\subsection{Variations on a theme}
\label{s:private}

A related question, asked by Itai Benjamini (private communication) is whether there exists some absolute constant $C>0$ such that for every finite group $G$ for all two symmetric sets of  generators $S$ and $S'$ such that $S \subset S'$ we have that the  mixing time of SRW on the  Cayley graph of $G$ with respect to $S'$ is at most  $C \frac{|S'|}{|S|} $ times the mixing time of SRW on the Cayley graph of $G$ with respect to $S$ (a set $S$ is called \emph{symmetric} if $S=S^{-1}\coloneqq \{s^{-1}:s \in S \}$). Our example also disproves this. %: indeed, the construction described in \S\ref{s:sketch} below satisfies %We give a negative answer. Furthermore, we show that the mixing time of SRW on $(G,S')$ may be of strictly larger order than that of SRW on $(G,S)$, even when
In fact, $S \subset S' \subseteq S^3$ and  $|S'|-|S| \le \sqrt{ |S|}$,  where  $S^{i}\coloneqq \{s_1\cdots s_i:s_1,\ldots,s_i \in S \}$ for $i \in \N$. The definition of an $(a,b)$-quasi-isometry (see \S\ref{s:qirob}) gives that if $S \subseteq S' \subseteq S^i$ then $\Cay(G,S)$ and $\Cay(G,S')$ are $(i,0)$-quasi-isometric. 

The reason that $|S'|-|S|\le\sqrt{|S|}$ is explained in the proof sketch section below --- both share a complete graph on some set $K$ with $|K|\asymp n$. Hence this complete graph has an order of $\asymp n^2$ edges, and there are only $o(n)$ additional edges. We could have increased $S_n$ by including in it all $|K|!$ permutation of the elements in $K$, while keeping $S_n' \setminus S_n$ the same set (of size $o(n)$), thus making $\frac{ |S_{n}'|-|S_n|}{|S_n|} $ tremendously smaller. 

%For a Markov chain $ (X_t)_{t \ge 0}$ with stationary distribution $\pi$ we denote by $\Pr_{x}$ (resp.\ $\Pr_{\mu}$ for a distribution $\mu$) the distribution of $(X_t)_{t \ge 0 }$, given that  $X_{0}=x$ (resp.\ $X_0 \sim \mu$). The choice of continuous-time or discrete-time lazy\footnote{Lazy means that the transition matrix $P$ satisfies $P(x,x) \ge 1/2$ for all $x$. In the discrete-time setup we focus on the lazy setup since non-lazy reversible chains may exhibit some "bad behavior" related to near bipartiteness. We wish to emphasize that this is not the cause of the behavior exhibited in our examples.} setup will throughout be either clear from context or irrelevant. The $\varepsilon$ total-variation (\textbf{TV}) is defined as
%\begin{equation}
%\label{eq: tau_p}
%\mix(\epsilon)\coloneqq  \min \{t: \max_x \|\Pr_x(X_t=\cdot)-\pi \|_{\TV}\le \epsilon \}, \quad \text{where} \quad \|\mu-\nu \|_{\TV}\coloneqq \frac 12 \sum_{x}|\mu(x)-\nu(x)|.
%\end{equation}
%When $\epsilon=1/4$ we omit it from the notation and terminology.\footnote{As generally $\mix((2\delta) ^{i}) \le i\mix(\delta) )$ for all $\delta \in (0,1/2)$ (e.g.\ \cite[Ch.\ 4]{cf:LPW}),  the choice of $1/4$ arbitrary.} We denote the TV \emph{mixing time} of SRW on a graph $G$ by $\mix(G)$. 

We will also be interested in weighted versions of the problem, as these allow us to define `weak' perturbations in a natural way. 
Let $\Gamma$ be a group %let $S$ be a symmetric set of generators
and let $W\coloneqq (w(s):s \in S)$ be symmetric weights (i.e.\ $w(s)=w(s^{-1})$) such that the support of $W$ generates $\Gamma$. The discrete-time lazy random walk on $\Gamma$ with respect to $W$ is the process with transition probabilities $P(g,g)=1/2$ and $P(g,gs)=\frac{w(s)}{2\sum_{r \in S}w(r)}$ for all $g,s \in \Gamma$. We denote its TV (total variation) mixing time by  $\mix(\Cay(\Gamma,W))$. In continuous time, let $R\coloneqq (r(s):s \in S)$ be symmetric rates. The continuous-time random walk on $\Gamma$ with respect to $R$ is the process that has infinitesimal transitions rates $r(s)$ between $g$ and $gs$ for all $g,s \in \Gamma$. Denote its mixing time by  $\mix(\Cay(\Gamma,R))$. %We omit $W$ from the notation when $w(s)=1$ for all $s \in S$ corresponding to lazy SRW, and similarly for $R$.
As in the unweighted case, due to the group symmetry the invariant distribution is uniform and the TV distance between it and the distribution of the walk at some given time is independent of the initial state.

Recall that $\sym_n$ is the symmetric group (the group of permutations of $n$ elements). %\textcolor{blue}{As usual, for $a,b: \mathbb{N} \to (0,\infty)$ we write  $a(n) \ll b(n)$ (as well as $a(n)=o(b(n))$) if $\lim_{n \to \infty}\frac{a(n)}{b(n)}=0$. }
The following is the promised weighted version of our main result.
\begin{maintheorem}
\label{thm:weighted}
%There exist sequences $S_n \subset S_n' \subset \mathfrak{S}_{n}$ of sets of transpositions  satisfying $S_n' \subseteq S_n^3 $, $|S_{n}'|-|S_n| \le \sqrt{|S_n|}$ and \eqref{e:comparisoncayley} such that%\footnote{We write $o(1)$ for terms which vanish as $n \to \infty$. We write $f_n=o(g_n)$ or $f_n \ll g_n$ if $f_n/g_n=o(1)$. We write $f_n=O(g_n)$ and $f_n \lesssim g_n $ (and also $g_n=\Omega(f_n)$ and $g_n \gtrsim  f_n$) if there exists a constant $C>0$ such that $|f_n| \le C |g_n|$ for all $n$. We write  $f_n=\Theta(g_n)$ or $f_n \asymp g_n$ if  $f_n=O(g_n)$ and  $g_n=O(f_n)$.}
%\begin{equation}
%\label{eq: 1}
%\mix(\mathfrak{S}_{N_n},S_{n}') \gtrsim \mix(\mathfrak{S}_{N_n},S_{n}) \log \log \log |\mathfrak{S}_{N_n}|.
%\end{equation}
%There exists some absolute constant $N_0\in \N$ such that
For every $f: \mathbb{N} \to [1,\infty)$ satisfying that $1 \ll f(n) \le \log \log \log n$, there exist   a sequence $(S_n)_{n =3}^{\infty}$ of sets of transpositions $S_n\subset\sym_{n}$  and a sequence of weights $(W_n)_{n=3}^{\infty}$,  such that
  %for all $n \ge N_0$ we have that
$W_n=(w_n(s))$ is supported on $S_n$ and satisfies that $1 \le w_{n}(s) \le 1+\left(f(n!)/\log\log n\right)^{1/4}$ for all $s \in S_n$, and such that
\begin{equation}
\label{eq: 2}
\mix(\Cay(\mathfrak{S}_{n},W_n)) \gtrsim \mix(\Cay(\mathfrak{S}_{n},S_{n}))f(n!).
\end{equation}
Similarly, in continuous time if we set $R_{n}=W_n$ (for the above $W_n$)  we get that
\begin{equation}
\label{eq: 2'}
\mix(\Cay(\mathfrak{S}_{n},R_n)) \gtrsim \mix(\Cay(\mathfrak{S}_{n},S_{n}))f(n!).
\end{equation} 
\end{maintheorem}
We remark that the power $1/4$ is not optimal (it was not a priority for us to optimise it). As before, $|S_n|\asymp n^2$. %\textcolor{blue}{ In the statement of Theorem \ref{thm:weighted} both $(S_n)_{n \ge N_0}$ and $(W_n)_{n \ge N_0}$ depend on $f$, however $N_0$ and the implicit constants in \eqref{eq: 2} and \eqref{eq: 2'} do not.}
\subsection{A non-transitive example}
Our third result shows that if one is willing to consider non-transitive instances, then indeed one can have a bounded degree example whose (usual worst-case) mixing time is of strictly smaller order than the mixing time starting from the best  initial state (i.e.\ the one from which the walk mixes fastest) after a small perturbation. In all previous constructions of graphs with a sensitive mixing time there was a large set that starting from it the walk mixes rapidly both before and after the perturbation, and the mixing time is governed by the hitting time of this set (which is sensitive by construction). In particular, the mixing time started from the best initial state is not sensitive.    

Let $G$ be a connected graph.  Let $W=(w(e):e \in E(G))$ be  positive edge weights. Consider the \emph{lazy random walk} $(X_k)_{k=0}^\infty$ on $G$ i.e.\ the process with transition probabilities $P(x,y)=\frac{w(xy)}{2\sum_{z}w(xz)} $ and $P(x,x)=\frac{1}{2}$ for all neighbouring $x,y \in G$. %We denote the corresponding TV mixing time by $\mix(G,W) $.
For $x \in G$ we define the mixing time starting from $x$ by
\[
\mix(G,W,x)\coloneqq \min \{k: \|\Pr_x(X_k = \cdot)-\pi \|_{\TV}\le 1/4 \}.
\]
With this definition the usual mixing time $\mix(G,W)$ (see \S\ref{s:mixingpre} below) is equal to $\max_x\mix(G,W,x)$.
%As before, we omit $W$ from the notation when $w(e)=1$ for all $e \in E$.  
\begin{maintheorem}
\label{thm: 2}
There exist a  sequence of finite graphs $L_n=(V_n,E_n)$ of diverging sizes and uniformly bounded degree (i.e.\ $\sup_n \max_{v \in V_n}\deg v<\infty$) and a sequence of   some symmetric edge weights $W_n=(w_n(e):e \in E_{n})$ such that $1 \le w_{n}(e) \le 1+\delta_n$ for all $e \in E_{n}$ and such that
\begin{equation}
\label{eq: 3}
\max_{x\in V_n}\mix(L_{n},1,x) \le \delta_n \min_{x \in V_n} \mix(L_{n},W_n,x).
\end{equation}
for some $\delta_n\to 0$.
\end{maintheorem}
It follows from Theorem \ref{thm: 2} that the average TV mixing time, by which we mean $\inf \{t:\sum_{x}\pi(x)\|\Pr_x(X_t = \cdot)-\pi \|_{\TV} \le 1/4 \}$, can be sensitive to perturbations. This is in contrast with the average $L_2$ mixing time (see \S\ref{s:mixingpre}). This gives a negative answer to a question of L.\ Addario-Berry (private communication).

As in Theorem \ref{thm:short}, the change in the order of the mixing time in Theorem \ref{thm: 2} (the inverse of the $\delta_n$ in \eqref{eq: 3}) is $o(\log \log \log |V_{n} |)$. If we replace the condition $w_n\le 1+\delta_n$ with $w_n\le 1+c$ then the change in the order of the mixing time can be as large as  $\log \log \log |V_{n} |$.

Let us quickly sketch the construction of Theorem \ref{thm: 2} (full details are in \S\ref{s:proofthm2}). Let $n$ be some number and let $S_n$ be the set of transpositions from Theorem \ref{thm:weighted}. Let $H$ be a large, fast mixing graph, and let $A$ be some subset of the vertices of $H$ with $|A|=|S_n|$ and with the vertices of $A$ far apart from one another. The graph $L$ of Theorem 3 has as its vertex set $\mathfrak{S}_n\times H$ (we are using here the same notation for the graph and its set of vertices). We choose the edges of $L$ such that random walk on $L$ has the following behaviour. Its $H$ projection is just simple random walk on the graph $H$. Its $\mathfrak{S}_n$ projection is also simple random walk on $\Cay(\mathfrak{S}_n,S_n)$, but slowed down significantly. Any given transposition $s\in S_n$ can be applied only when a corresponding vertex of $A$ is reached in the second coordinate. The perturbation goes by perturbing only the $\mathfrak{S}_n$ projection. We defer all other details to \S\ref{s:proofthm2}.

\subsection{A proof sketch}\label{s:sketch}
We will now sketch the proof of our main result, Theorem \ref{thm:short} (the proof of Theorem \ref{thm:weighted} is very similar). Readers who intend to read the full proof can safely skip this section.

Random walk on $\Cay(\sym_{n},S_n)$ with $S_n$ composed of transposition is identical to the \emph{interchange process} on the graph $G$ which has $n$ vertices and $\{x,y\}$ is an edge of $G$ if and only if the transposition $(x,y)\in S_n$. Hence we need to construct two graphs $G$ and $G'$ on $n$ vertices, estimate the mixing time of the two interchange processes and show that the corresponding Cayley graphs are quasi-isometric.

Our two graphs have the form of `gadget plus complete graph'. Namely, there is a relatively small part of the graph $D$ which we nickname `the gadget' and all vertices in $G\setminus D$ are connected between them. While $D$ and the corresponding $D'$ in $G'$ will be small (we will have $|D|=|D'|$), they dominate the mixing time of the interchange process.

To describe the gadget, let $u\in\bbN$ and $\eps\in(0,\frac12)$ be some parameters. The gadget will have $u$ `stages' $H_1,\dotsc,H_u$ (the gadget is almost $\cup_{i=1}^u H_i$ but not quite). We obtain each $H_i$ by `stretching' the edges of some graph $H_i'$ which is a union of binary trees of depth $s_i\coloneqq  4^{i-1}u$ (note that $H_i'$ has the depth exponential in $i$ and hence has volume doubly exponential in $i$). To get $H_i$, replace each edge of $H_i'$ with a path of length $\ell_i\coloneqq  2^{u+1-i}$. Namely for each edge $\{x,y\}$ of $H_i'$ we add $\ell_i-1$ new vertices (denote them by $v_1,\dotsc,v_{\ell_i-1}$, and denote also $v_0=x$ and $v_{\ell_i}=y$) and connect $v_j$ to $v_{j+1}$ for all $j\in\{0,\dotsc,\ell_i-1\}$; and remove the edge $\{x,y\}$.

We still need to explain how many trees are in each $H_i$ and how they are connected to one another and to the rest of the graph. For this we need the parameter $\eps$, which at this point can be thought of as a sufficiently small constant. For each of the vertices in each of the trees (before stretching) we label the children arbitrarily `left' and `right'. For each leaf $x\in H_i$ we define $g(x)$ to be the number of left turns in the (unique) path from the root to $x$. We now let
\begin{equation}\label{eq:defBi sketch}
B_i\coloneqq \Big\{x\textrm{ leaf of }H_i:g(x)>\left(\sdfrac12+\eps\right)s_i\Big\}.
\end{equation}
The sets $B_i$ are used twice. First we use them to decide how many trees will be in each $H_i$. For $i=1$ we let $H_1$ be one tree. For every $i>1$, we let $H_i$ have $|B_{i-1}|$ trees, and identify each point of $B_{i-1}$ with one of the roots of one of the trees in $H_i$. Second, we use the $B_i$ to connect the $H_i$ to the complete graph. Every leaf of $H_i$ which is not in $B_i$ is identified with a vertex of the complete graph (the complete graph $K$ will be of size $n-o(n)$, much larger than $\cup_{i=1}^u H_i$ which will be of size $O(n^{1/4})$, and so  most of the vertices of $K$  are not identified with a vertex of the gadget). This terminates the construction of $G$.  See Figure \ref{fig:gadget}. Experts will clearly notice that this is a variation on the perturbed tree idea. In other words, while the perturbed tree itself (as noted above) is highly non-transitive, one can use it as a basis for transitive example by examining the interchange process on it.

\begin{figure}
\centering{\input{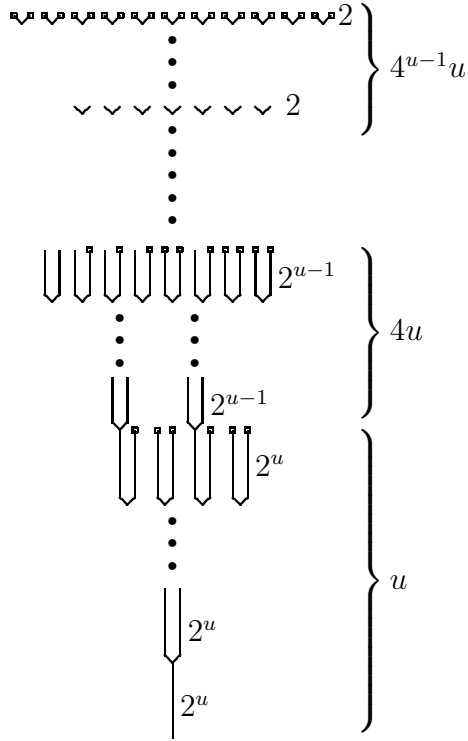}}
\caption{The gadget. Vertices marked with small squares actually belong to the complete graph rather than to the gadget.}\label{fig:gadget}
\end{figure}

The graph $G'$ is almost identical, the only difference is that in each path corresponding to a left turn we add short bridges. Namely, examine one such path and denote its vertices $v_0,\dotsc,v_{\ell_i}$ as above. Then in $G'$ we add edges between $v_{2j}$ and $v_{2j+2}$ for all $j\in\{0,\dotsc,\ell_i/2-1\}$. 

Why this choice of parameters? It is motivated by a heuristic that for such graphs, namely a gadget connected to a large complete graph, the mixing time of the interchange process is the time \emph{all} particles have left the gadget (they do not have to all be outside the gadget at the same time, it is enough that each particle left the gadget at least once by this time). See \S\ref{s:Oliveira} below for some context for this heuristic. Thus we are constructing our $H_i$ such that the time that it takes all particles to leave $H_i$ is approximately independent of $i$. Indeed, the time a particle takes to traverse a single stretched edge is approximately $\ell_i^2 \asymp4^{u-i}$ while each tree of $H_i$ has depth $4^{i-1}u$ (in the sense that this is the depth of the tree before its edges have been stretched) and the particle has to traverse all levels of $H_i$, so it exits $H_i$ after time approximately $4^{u-1}u$, which is independent of $i$. And this holds for all particles simultaneously because the probability that a particle takes $\lambda\cdot 4^{u-1}u$ time to traverse the tree (for some $\lambda>1$) is exponentially small in the number of layers $4^{i-1}u$, and hence that would not happen to any of the particles in the tree, which has approximately $2^{4^{i-1}u}$ particles, if $\lambda$ is sufficiently large. In the roughest possible terms, the growing height of the trees is dictated by the growing number of vertices (which must grow because $H_i$ has many more roots than $H_{i-1}$, since each $x\in B_{i-1}$ is a root of $H_i$) while the decreasing stretching balances the growing height to get approximately uniform expected exit time. The only exception is $H_1$, whose height is not dictated by the number of roots (clearly, as there is only one), but by the stretching.

With the definitions of $G$ and $G'$ done, estimating the mixing times is relatively routine, so we make only two remarks in this quick sketch. How do we translate the fact that all particles visited the complete graph into an upper bound on the mixing time? We use a \emph{coupling} argument. We couple two instances $\sigma$ and $\sigma'$ of the interchange process (in continuous time) using the same clocks and letting them walk identically unless $\sigma(x)=\sigma'(y)$ for an edge $\{x,y\}$ that is about to ring, in which case we apply the transposition to exactly one of $\sigma$ or $\sigma'$, reducing the number of disagreements (this coupling involves a standard trick of doubling the rates, and censoring each step with probability 1/2). The fact that the complete graph is much larger and has many more edges simplifies our analysis (the reader can find the details of the coupling in \S \ref{s:coupling}).

The lower bound for the mixing time on $G'$ uses the standard observation that adding those edges between $v_{2j}$ and $v_{2j+2}$ makes the left turn more likely to be taken than the right turns, transforming $B_i$ from an atypical set (with respect to the hitting distribution of the leaf set of $H_i$) to a typical one, and hence the particle which started at the root of $H_1$ has high probability to traverse \emph{all} $H_i$ before entering the complete graph for the first time. This, of course, takes it $4^{u-1}u^2$ time units (compare to the mixing time bound of $4^{u-1}u$ for the interchange process on $G$). Of course, the mixing time of the interchange process on $G'$ is also bounded by the time all particles leave the gadget, but we found no way to use this. We simply bound the time a single particle leaves the gadget and get our estimate.

\subsection{The mixing time of the interchange process}\label{s:Oliveira}
Since our proof revolves around estimating the mixing time of the interchange process on some graph, let us spend some time on a general discussion of this topic. %The ***check with an English speaker*** proof proper will start at \S\ref{s:uandepsilon}.
We first mention some conjectures relating the mixing time of the interchange process on a finite graph $G$ to that of $|G|$ independent random walks on $G$. 

Given a finite graph $G=(V,E)$ and edge rates $R$ the corresponding $n$-fold product chain is the continuous-time Markov chain on $V^n$ satisfying that each coordinate evolves independently as a random walk on $G$ with edge rates $R$. This is a continuous-time  walk on the $n$-fold Cartesian product of $G$ with itself, whose symmetric edge rates $R_n$ are given by
  \[
  R_n((v_1,\dotsc,v_n),(v_1,\dotsc,v_{k-1},v_k',v_{k+1},\dotsc,v_n)):=R(v_k,v_k')
  \]
for all $v_1,\ldots,v_n,v_k' \in V$ and $k \in [n]$.
We shall refer to this Markov chain as $n$ independent random walks on $G$ with edge rates $R$ and denote its (TV) mixing time by $ \mix(n \text{ independent RWs on } G,R)$. As usual, the mixing time is defined with respect to the worst starting tuple of $n$ points, which turns out to be when they all start from the worst point for a single walk on $G$ with edge rates $R$.

Oliveira \cite{Olive} conjectured that there exists an absolute constant $C>0$ such that the TV mixing time of the interchange process on an $n$-vertex graph $G$ with rates $R$, i.e.\ $\mix(\Cay(\sym_n,R))$, is at most $C \mix(n \text{ independent RWs on }G,R)$. %$n$ independent random walks on $(G,R)$
  See \cite[Conjecture 2]{HS} and \cite[Question 1.12]{Hex} for two different strengthened versions of this conjecture. See \cite{HS} for a positive answer for high dimensional products. 

For the related \emph{exclusion} process, some progress on Oliveira's conjecture is made in \cite{Hex}. Returning to the interchange process, in the same paper the following more refined question is asked \cite[question 1.12]{Hex}: Is $\mix(\Cay(\sym_n,R))$ equal up to some universal constants to the mixing time of $n$ independent random walks on $(G,R)$ starting from $n$ \emph{distinct} locations? (See \cite{Hex} for precise definitions). We see that our result is related to finding some graphs $G$ such that the mixing time of $|G|$ independent SRW with edge rates $1$ on $G$, starting from distinct initial locations, is sensitive under small perturbations. In fact, the graphs we construct in this paper satisfy this property too, but in the interest of brevity we will not prove this claim (the proof is very similar to the one for the interchange process we do provide). This conjectured relation between the exclusion process and independent random walks is behind the heuristic we employed (and mentioned in \S\ref{s:sketch}) to construct our example. 

As we now explain,  if we did not require the initial locations to be distinct (as is the case in Oliveira's conjecture) such sensitivity could not occur. 
It is easy to show (e.g., \cite{Hex}) that when $|G|=n$
\[
  \frac 14  \rel(G,R) \log n
  \le \mix(n \text{ independent RWs on }G,R)
  \le 4  \rel(G,R) \log n ,
  \] where $\rel(G,R)$ is the \emph{relaxation time} of $(G,R)$, defined as the inverse of the second smallest eigenvalue of $-\mathcal{L}$, where $\cL$ is the infinitesimal Markov generator of the walk $(G,R)$. %\footnote{\tb{[MAYBE ADD A REFERENCE TO OCTOPI PAPER] and say it gives an affirmative answer to Oliveira's question in many cases when $\rel(G) \asymp \mix(G)$.}}
  The relaxation time is robust under small perturbations (see \S\ref{s:com}), and hence so is $ \mix(n \text{ independent RWs on }G,R)$. Our result that the mixing time is sensitive does not contradict Oliveira's conjecture, as he conjectured only an upper bound (which, in our case, is sharp for neither $S_n$ nor $S_n'$).

Loosely speaking, in order to make the mixing time of $n$ independent random walks starting at distinct locations of smaller order than (the robust quantity) $\rel(G,R) \log n$ it is necessary that the eigenvector corresponding to the minimal eigenvalue of $-\cL$ be localised on a set of cardinality $n^{o(1)}$. This is a crucial observation in tuning the parameters in our construction, which explains why for smaller areas of the graph (namely, $H_i$ with small index $i$) we `stretch' edges by a larger factor. This is the opposite of what is done in \cite{unifsensitivity}. 

%The moral is that even given the idea of considering the interchange process on a `bad' base graph, the choice of parameters in the construction is very subtle and the heuristic above is extremely helpful in making educated guesses. %(see Remark \ref{r:subtle} for more on this point).
%\footnote{This is similar to how some heuristics concerning the spectral-profile and the relation between hitting times and the $L_2$ mixing time established in \cite{cf:L2} are used to correctly tune the parameters of the construction in \cite{unifsensitivity}.} 

Lastly, we comment that in contrast with a single random walk, in order to change the mixing time of $n$ independent random walks, starting from $n$ distinct initial locations, it does not suffice for the perturbation only to change the typical behavior of the walk, but rather it is necessary that it significantly changes the probabilities of some events in some sufficiently strong quantitative manner. See \cite{unifsensitivity} for a related discussion, about why it is much harder to construct an example where the uniform mixing time is sensitive than it is to construct one where the TV mixing time is sensitive.

\subsection{Quasi-isometries and robustness}\label{s:qirob}
\label{s:QI}
Since we hope this note will be of interest to both group theory and Markov chain experts, let us take this opportunity to compare two  similar notions related to comparison of the geometry of two graphs or of two reversible Markov chains. The first is the notion of quasi-isometry which is more geometric in nature. The second is the notion of robustness which is more analytic. In particular, we are interested in properties which are preserved by these notions.

This discussion is an important part of the background, but let us advise the readers that it is not necessary to appreciate our results, as they apply in both cases. For example, Theorem \ref{thm:short} shows that the mixing time is neither quasi-isometry invariant nor robust.

  A quasi-isometry (defined first in \cite{G81}) between two metric spaces $X$ and $Y$ is a map $\phi:X\to Y$ such that for some numbers $(a,b)$ we have
\[
\forall u,v\in X \qquad \frac{d(u,v)-b}{a}\le d(\phi(u),\phi(v))\le ad(u,v)+b
\]
where $d$ denotes the distance (in $X$ or in $Y$, as appropriate). Further we require that for every $y\in Y$ there is some $x\in X$ such that $d(\phi(x),y)\le a+b$. We say that $X$ and $Y$ are $(a,b)$ quasi-isometric if such a $\phi$ exists. (Our choice of definition is unfortunately only partially symmetric. If $\phi:X\to Y$ is an $(a,b)$ quasi-isometry then one may construct a quasi-isometry $\psi:Y\to X$ with the same $a$ but perhaps with a larger $b$.)

For a property of random walk that is defined naturally on infinite graphs we say that it is \emph{quasi-isometrically invariant} if whenever $G$ and $H$ are two quasi-isometric infinite graphs, the property holds for $G$ if and only if it holds for $H$ (the graphs are made into metric spaces with the graph distance). Examples include a heat kernel on-diagonal upper bound of polynomial type \cite{CKS87}, an off-diagonal upper bound \cite{G92} and a corresponding lower bound \cite{GHL09,B09}. A particularly famous example is the Harnack inequality \cite{BM18}. For a quantitative property of random walk naturally defined on finite graphs, such as the mixing time, one says that it is invariant to quasi-isometries if, whenever $G$ and $H$ are $(a,b)$-quasi-isometric, the property may change by a constant that depends only on $a$ and $b$ and not on other parameters. Similar notions may be defined for Brownian motion on Riemannian manifolds, and one may even ask questions like `if a manifold $M$ is quasi-isometric to a graph $G$ and Brownian motion on $M$ satisfies some property, does random walk on $G$ satisfy an equivalent property?' and a number of examples of this behaviour are known.

The notion of \emph{robustness} does not have a standard definition, and in particular the definitions in \cite{cf:Ding} and \cite{unifsensitivity} differ (and also differ from the definition we will use in this paper). Nevertheless they all have a common thread: a definition for Markov chains that implies that the property in question is preserved under quasi-isometry of graphs of bounded degree, but that makes sense also without any a priori bound on the transition probabilities. Here we will use the following definition. Let $\cM$ be the set of finite state Markov chains. We say that a $q:\cM\to[0,\infty]$ is \emph{robust} if for any $A \in (0,1] $ there exists some $K \in (0,1] $ such that the following holds. %it can vary by at most some factor $\delta(c)>0$, depending only on $c$, between
Assume $M$ and $M'$ are two irreducible reversible Markov chains on the same finite state space $V$ with stationary distributions $\pi$ and $\pi'$ and transition matrices $P$ and $P'$ satisfying
\begin{equation}\label{e:DFcomparison}
  \begin{aligned}
  \forall \, x &\in V,& A\pi'(x)  &\le \pi(x)\le \tfrac1A\pi'(x) \\
  \forall \, f &\in \R^V,& A\mathcal{E'}(f,f)&\le \mathcal{E}(f,f)\le  \tfrac1A\mathcal{E'}(f,f)
  \end{aligned}
\end{equation}
where $\mathcal{E}(f,f)$ and $\mathcal{E}'(f,f)$ are the corresponding Dirichlet forms, namely
\[
\mathcal{E}(f,g)\coloneqq \frac 12\sum_{u,v \in V}\pi(u)P(u,v)(f(u)-f(v))(g(u)-g(v))
\]
and similarly for $\mathcal{E}'$. Then $q(M)\ge Kq(M')$.

We also define robustness for Markov chains in continuous time, and in this case we replace $P(u,v)$ above with $\cL(u,v$) which is the infinitesimal rate of transition from $u$ to $v$, but otherwise the definition remains the same. 

If $P$ and $P'$ are simple random walks on $(a,b)$ quasi-isometric graphs with the same vertex set (with the quasi-isometry being the identity), whose maximal degrees are at most $D$, then \eqref{e:DFcomparison} holds with some $A$ depending only on $(a,b,D)$ \cite{diaconiscomparison}. Thus a robust quantity is also quasi-isometry invariant between graphs of bounded degree on the same vertex set. %It is also worth noting that currently there is no interesting example of a property of random walk which is quasi-isometrically invariant but not robust, or vice versa.

Each notion has its advantages and disadvantages relative to the other notion. Quasi-isometry has the flexibility that the spaces compared need not be identical or even of the same type, indeed the fact that a Lie group (a continuous metric space, indeed a manifold) is quasi-isometric to any cocompact lattice of it (a discrete metric space) plays an important role in group theory. Robustness has the advantage that unbounded degrees are handled seamlessly.

Returning to our results, since the examples of our Theorem \ref{thm:short} are not of bounded degree, it is natural to ask if they satisfy a comparison of Dirichlet form of the form \eqref{e:DFcomparison}. In fact, this is true because in said examples our pair of sets of generators $S_n$ and $S_n'$ (from the statement of Theorem \ref{thm:short}) satisfy for all $n$ that $S_n \subset S_n'$ and that any
$s' \in S_n' \setminus S_n$ can be written as $s_{1}(s')s_2(s')s_3(s') \in S_n^3=\{xyz:x,y,z \in S_n \} $ in a manner satisfying that 
\begin{equation}
\label{e:comparisoncayley}
\max_{s \in S_n} \sum_{s'\in S_n' \setminus S_n} \sum_{i=1}^3 \mathbbm{1}\{s_i(s')=s \}
\le 2. 
\end{equation}
It is standard and not difficult to see that \eqref{e:comparisoncayley} implies the comparison of Dirichlet forms condition \eqref{e:DFcomparison} (see e.g.\ \cite[Theorem 4.4]{bere}). %\textcolor{blue}{with $c=\frac{1}{1+3 \times 2}=\frac 17$ (see e.g.\ \cite[Theorem 4.4]{bere}; we included the middle expression $\frac{1}{1+3 \times 2}$ to emphasize how \cite[Theorem 4.4]{bere} is being used here)}.
Thus, the examples of Theorem \ref{thm:short} also satisfy \eqref{e:DFcomparison} with  $A$ being a universal constant. We remark that in general   $S \subset S' \subseteq S^3$ is sufficient for deriving \eqref{e:DFcomparison} only with an $A$ that may depend on  $|S'|$.

\subsection{Remarks and open problems}
We start with a remark on the Liouville property problem, a problem which for us was a significant motivation for this work. An infinite graph with finite degrees is called Liouville if every bounded harmonic function is constant (a function $f$ on the vertices of a graph is called harmonic if $f(x)$ is equal to the average of $f$ on the neighbours of $x$ for all $x$).

An open problem in geometric group theory is whether the Liouville property is quasi-isometry invariant in the setup of Cayley graphs (and, in the spirit of the aforementioned question of Benjamini, whether it is preserved under deletion of some generators, possibly by passing to a subgroup, if the smaller set of generators does not generate the group). The problem of stability of the Liouville property is related to that of mixing times. Indeed,  the example of T.\ Lyons \cite{L87} mentioned above which is a base for all previous examples for sensitivity was in fact an example for the instability of the Liouville property (for non-transitive graphs).

A result of Kaimanovich and Vershik (see \cite{KV83} or \cite[Ch.\ 14]{Lyons}) states that for Cayley graphs, the Liouville property is equivalent to the property of the walk having zero speed. Of course, our graphs being finite means there is no unique number to be designated as `speed', as in the Kaimanovich-Vershik setting. But still it seems natural to study the behaviour of $\textrm{dist}(X_t,1)$ as a function of $t$, where $X_t$ is the random walk, 1 is the identity permutation (and the starting point of the walker), and dist is the graph distance with respect to the relevant Cayley graph (with respect to $S_n$ or $S_n'$, as the case may be). Interestingly, perhaps, the functions increase linearly for the better part of the process for both our $S_n$ and $S_n'$, so we cannot reasonably claim we show some version of instability for the speed for finite graphs. (We will not prove this claim, but it is not difficult.)
%A priori, it seems like the mixing time is a more complicated quantity than the speed of the walk, suggesting that the problem of stability of mixing times in the transitive setup could be at least as hard as the stability of the Liouville property problem. However, in our example, the mixing time of the walk is sensitive, while the speed of the walk is not. The latter is positive both before and after the perturbation, for the better part of the process (we will not prove this claim, but it is not difficult). %until a certain point where the finiteness of the graph no longer allows the walk to continue advancing at a positive speed.

Due to the relation to the Liouville problem, there is interest in reducing the degrees in Theorem \ref{thm:short}. We note that since our $S_n$ is a set of transpositions, we must have $|S_{n}| \le \binom{n}{2}\asymp \left(\frac{\log | \mathfrak{S}_{n}|}{\log \log |S_{n}|}\right)^{2}$.  As explained in the proof sketch section above, in our construction, there is a set $K \subset [n]\coloneqq  \{1,\ldots,n \} $ such that $|K| = n(1-o(1))$ and all of the transpositions of the form $(a,b)$ with $a,b \in K$ belong to $S_n$. Hence  $|S_n| \asymp n^2$.  

Let us mention two possible approaches to reduce the size of $S_n$. The first is to replace the complete graph over $K$ in the construction by an expander. In this case we will have $|S_{n}| \asymp n$. It seems reasonable that this approach works, but we have not pursued it. Let us remark at this point that the mixing time of the interchange process on an expander is not known, with the best upper bound being $\log^2 n$ \cite{AK} (see also \cite{Hex}).

The second, and more radical, is to replace the $\binom{|K|}{2}$ transpositions corresponding to pairs from $K$ by some number (say $m$, but importantly independent of $n$) of random permutations of the set $K$, obtained by picking  $m$ independent random perfect matchings of the set $A$, and for each perfect matching taking the permutation that transposes each matched pair. (If $|K|$ is odd, we keep one random element unmatched.) Note that the Cayley graph is no longer an interchange process, and that approximately $n^2$ elements have been replaced by a constant number. The degree would still be unbounded because of the other part of the graph. Again, we did not pursue this approach. One might wonder if it is possible to replace the entire graph, not just $K$, by matchings, but this changes the mixing time significantly.

%Another possibility to reduce the size of $S_n$ is to replace  the complete graph in the construction by an expander. In this case we will have $|S_{n}| \asymp n$. It seems reasonable that either approach works, but we have not pursued these directions. Let us remark at this point that the mixing time of the interchange process on an expander is not known, with the best upper bound being $\log^2 n$ \cite{AK} (see also \cite{Hex}).
%We believe that this will  exhibit the same behaviour. It would require considerable effort to prove that this is indeed the case (and we do not know how to do so\footnote{Although it seems within reach, especially when replacing the number os `stages' in the construction from $u$ to, say $\log u$, which results in a smaller change to the mixing time than in Theorem \ref{thm: 1}.}).
%The jury is still out on the following question, in which we allow  $\sym_{n}$ to be replaced by any sequence of finite groups $G_n$ of diverging sizes (in particular, $S_n$ need not be a set of transpositions) and  $\log \log \log |G_n|$ from \eqref{eq: 1} to be replaced by any other diverging function of $|G_n|$.

\begin{question}
\label{q:1}
Can one take the set of generators $S_n$ to be of constant size? (Certainly, not with transpositions but with general subsets of $\sym_n$, or with other groups). If not, can one take $|S_n|$ to diverge arbitrarily slowly as a function of $|G_n|$? Is there a relation between the degree of the graph and the maximal amount of distortion of the mixing time which is possible?
\end{question}

%We conclude the introduction with some open problems:
%The following question is only open in the case that the degree is unbounded.
A related question is the following.
\begin{question}
\label{q:2}
Does the aforementioned question of Benjamini have an affirmative answer for \emph{bounded degree} Cayley graphs?
\end{question}
%The same can be asked about nilpotent groups of uniformly bounded step.
Here are two questions about the sharpness of our $\log\log\log$ term.
\begin{question}
\label{q:3}
Does there exist a sequence of finite groups $G_n$ of diverging sizes, and sequences of generators $S_n \subset S_n' \subseteq S_n^i$ for some $i \in \N$ (independent of $n$) for all $n$, such that $|S_n'| \lesssim |S_n|$ and
\begin{equation}
\label{e:q15}
\mix(G_{n},S_{n}') \gtrsim \mix(G_{n},S_{n}) \log |G_{n}|?
 \end{equation}
\end{question}
 
\begin{question}
\label{q:4} Can one have in the setup of Theorem \ref{thm: 2}
\begin{equation}
\label{e:q16}
\min_{x \in V_n} \mix(G_{n},W_n,x) \gtrsim \mix(G_{n})\log |V_n|  ? 
\end{equation}
\end{question}
The opposite inequalities to \eqref{e:q15} and \eqref{e:q16} hold since the spectral gap  is a quasi isometry invariant (see \S\ref{s:com}) and on the other hand determines the mixing time of a random walk on an $n$-vertex graph up to a factor $2 \log n$ (see, e.g., \cite[\S12.2]{cf:LPW}).%\footnote{Cf.\ \cite{HP}, where it is shown that in the non-transitive setup, bounded change to the edge weights can indeed change the (usual worst-case) TV mixing time of a random walk on a bounded degree graph $G=(V,E)$ by a factor of order $\log |V|$.}

Our last question pertains to Theorem \ref{thm:L2mix} below. It is inspired by a question of Itai Benjamini on the Liouville property in the infinite setting.
\begin{question}
Let $G=(V,E)$ be a finite connected vertex-transitive graph. Is the uniform (or $L^2$) mixing time robust under bounded perturbations of the edge weights? (certainly, this is open only when the perturbation does not respect the transitivity). Likewise, does there exist some $C(a,b,d)>0$ (independent of $G$) such that if the degree of $G$ is $d$ and $G'$ is $(a,b)$-quasi-isometric to $G'$ (which, again, need not be vertex-transitive), then the uniform mixing times  of the SRWs on the two graphs can vary by at most a $C(a,b,d)$ factor?
\end{question}

%\label{s: related}
We end the introduction with a few cases for which the mixing time is known to be robust. Robustness of the TV and  $L_{\infty}$  mixing times for all reversible Markov chains under changes to the holding probabilities (i.e.\ under changing the weight of  each loop by at most a constant factor) was established in \cite{PS} by Peres and Sousi and in \cite{cf:L2} by J.H.\ and Peres. Boczkowski, Peres and Sousi \cite{BPS} constructed an example demonstrating that this may fail without reversibility. Robustness of the TV and  $L_{\infty}$  mixing times for
general (weighted) trees under bounded perturbations of the edge weights was established in \cite{PS} by Peres and Sousi and in \cite{cf:L2} by J.H.\ and Peres. Robustness of  TV mixing times for
general trees under quasi-isometries (where one of the graphs need not be a tree, but is `tree-like' in that it is quasi-isometric to a tree) was  established in \cite{berry} by Addario-Berry and Roberts.

In many cases known robust quantities provide upper and lower bounds on the mixing time which are matching up to a constant factor. For example, in the torus $\{1,\dotsc,\ell\}^d$ with nearest neighbour lattice edges the mixing time is bounded above by the isoperimetric profile bound on the mixing time \cite{cf:Evolving} and  below by the inverse of the spectral gap. For a fixed $d$ both bounds are $\Theta(\ell^2)$. As  both quantities are robust, we get that any graph quasi-isometric to the torus would have mixing time $\Theta( \ell^2)$, as in the torus. In fact,  the same holds for bounded degree Cayley graphs of moderate growth (see e.g.\ \cite[\S7]{Hex}). Moderate growth is a technical condition, due to Diaconis and Saloff-Coste \cite{DSC:moderate}, who determined the order of the mixing time and the spectral gap for such Cayley graphs. Breuillard and Tointon \cite{BT} showed that for Cayley graphs of bounded degree this condition is equivalent in some precise quantitative sense to the condition that the diameter is at least polynomial in the size of the group. 

Lastly, in a recent work \cite{LW} Lyons and White showed that for finite Coxeter systems increasing the rates of one or more generators does not increase the $L_p$ distance between the distribution of the walk at a given time $t$ and the uniform distribution for any $p \in [1,\infty]$. Since multiplying all rates by exactly a factor $C$ changes the mixing time by exactly a factor $1/C$, this implies that the mixing time is robust under bounded permutations of the rates of the generators.

\subsection{Notation}
\label{s:notation}
We denote $[n]=\{1,\dotsc,n\}$. We denote by $\mathbb{P}_v$ probabilities of random walk starting from $v$, which should be a vertex of the relevant graph. We denote by $c$ and $C$ arbitrary positive universal constants which may change from place to place. We will use $c$ for constants which are small enough and $C$ for constants which are large enough. We will occasionally number them for clarity. We denote $X\lesssim Y$ for $X\le CY$ and $X\asymp Y$ for $X\lesssim Y$ and $Y\lesssim X$. We denote $X\ll Y$ for $X=o(Y)$. Throughout we do not distinguish between a graph $G$ and its set of vertices, denoting the latter by $G$ as well. The set of edges of $G$ will be denoted by $E(G)$.

\section{Preliminaries}
\label{s: Pre}
\begin{definition} Let $\Gamma$ be a finitely generated group and let $S$ be a finite set of generators satisfying $s\in S\iff s^{-1}\in S$. We define the Cayley graph of $\Gamma$ with respect to $S$, denoted by $\Cay(\Gamma,S)$, as the graph whose vertex set is $G$ and whose edges are
  \[
  \{(g,gs):g\in \Gamma, s\in S\}.
  \]
\end{definition}
\begin{definition} Let $G$ be a weighted graph and let $(r(e)_{e\in E(G)})$ be the weights. The interchange process on $G$ is a continuous-time process in which particles are put on all vertices, all different. Each edge $e$ of $G$ is associated with a Poisson clock which rings at rate $r(e)$. When the clock rings, the two particles at the two vertices of $e$ are exchanged.
\end{definition}
The interchange process is always well defined for finite graphs (which is what we are interested in here). For infinite graphs there are some mild conditions on the degrees and on $r$ for it to be well defined. The interchange process on a graph $G$ of size $n$ is equivalent to a random walk in continuous time $X_t$ on $\sym_n$ with the generators $S$ of $\sym_n$ being all transpositions $(xy)$ (in cycle notation) for all $(xy)$ which are edges of $G$. The rate of the transposition $(xy)$ is $r(xy)$. The position of the $i^{\textrm{th}}$ particle at time $t$ is then $X_t^{-1}(i)$, where the inverse is as permutations. 
\subsection{Comparison of Dirichlet forms}
\label{s:com}
%% We say a quantity is robust if for any $c>0$ it can vary by at most some factor $\delta(c)>0$, depending only on $c$, between any two irreducible reversible Markov chains on the same finite state space $V$ with stationary distributions $\pi$ and $\pi'$ and transition matrices $P$ and $P'$ (resp.\ in continuous-time, infinitesimal Markov generators $\cL$ and $\cL'$) satisfying
%%  \begin{equation}
%%  \label{e:DFcomparison2}
%% \forall \, x \in V,\; \; c\pi'(x)  \le \pi(x)\le \pi'(x)/c \quad \text{and} \quad c \mathcal{E'}(f,f)\le \mathcal{E}(f,f)\le  \mathcal{E'}(f,f)/c \; \text{ for all }f \in \R^V,
%% \end{equation}
%%  where $\mathcal{E}(f,f):=\frac 12\sum_{v,u \in V}\pi(v)P(v,u)(f(v)-f(u))^2 $ and $\mathcal{E}'(f,f):=\frac 12\sum_{v,u \in V}\pi(v)P'(v,u)(f(v)-f(u))^2$ (resp., in continuous-time replace $P$ by $\cL$ and $P'$ by $\cL'$) are the \emph{Dirichlet forms} of the two chains. Indeed \eqref{e:DFcomparison} holds if $P'$ is obtained from $P$ as a result of changing the edge weights by at most a factor $1/c$ (edge-wise). Similarly, if $P$ and $P'$ are simple random walks on $(a,b)$ quasi-isometric graphs with the same vertex set, whose maximal degrees are at most $D$, then \eqref{e:DFcomparison} holds with some $c$ depending only on $(a,b,D)$ \cite{diaconiscomparison}.

Recall the condition \eqref{e:DFcomparison} for comparison of Dirichlet forms. When it holds %then by taking $f$ to be the indicator of the set $A$ we see that $\frac{1}{\pi(A)}\sum_{a \in A,b \notin A}\pi(a)P(a,b)$ and  $\frac{1}{\pi'(A)}\sum_{a \in A,b \notin A}\pi'(a)P'(a,b)$ can vary by at most a factor $c^2$ for all $A \subset V$, which implies a comparison of the isoperimetric profiles \cite{cf:Evolving} of the two chains. Condition \eqref{e:DFcomparison} also yields a similar comparison of the spectral-profiles \cite{cf:Spectral}. It also implies that their log-Sobolev constants differ by at most a factor $c^2$ \cite[Lemma 3.3]{diaconis}. Lastly, it also
it implies a comparison of the eigenvalues:
If $0=\lambda_1 \le \lambda_2 \le \cdots \le \lambda_{n}$ and  $0=\lambda_1 '\le \lambda_2 '\le \cdots \le \lambda_{n}'$  are the eigenvalues of $I-P$ and $I-P'$, respectively, then under \eqref{e:DFcomparison} (see e.g.\ \cite[Corollary 4.1]{bere} or \cite[Corollary 8.4]{cf:Aldous})
\begin{equation}
\label{e:lailai'}
A \lambda_i \le \lambda_i' \le \lambda_i/A \quad \text{for all } i. 
\end{equation}
The same inequality holds for the eigenvalues of the Markov generators $-\cL$ and $-\cL'$ in continuous time (that is, $\cL(x,y)=r(xy)$ for $x \neq y$ and $\cL(x,x)=-\sum_{y:\, y \neq x}r(xy)$, where $r(xy)$ is the rate of the edge $(xy)$ and with the convention that $r(xy)=0$ if $xy \notin E$).
% This follows from  the extremal characterization of the eigenvalues in terms of the Dirichlet form (e.g.\ \cite[Theorem 4.2]{bere}).
The proof is the same as in the discrete case (see, again, \cite[Corollary 4.1]{bere}). % alternatively, the analog of \eqref{e:lailai'} for the eigenvalues of $\mathcal{L}$ can be derived from the discrete-time case since in the finite state space reversible setup one can always write $\mathcal{L}=2\min_{x}\cL(x,x)(I-P)$, where $P$ is a reversible transition matrix satisfying $\min_y P(y,y) \ge 1/2$).
The quantity $\lambda_2$ is called the \emph{spectral gap}. It follows that it is robust. 

\subsection{Mixing times}
\label{s:mixingpre}
We now define the relevant notions of mixing: total variation, $L_2$ and uniform. We start with the total variation mixing time which is the topic of this paper, and which we will simply call \emph{the} mixing time.
\begin{definition}
Let $X_t$ be a Markov chain on a finite state space (in continuous or discrete time) with stationary measure $\pi$, and denote the probability that $X_t=y$ conditioned on $X_0=x$ by $P_t(x,y)$. Then the mixing time is defined by
\[
\mix=\max_x\inf\{t\ge 0:||P_t(x,\cdot\,)-\pi||_{\mathrm{TV}}\le\tfrac14\}.
\]
\end{definition}
%***do we? It seems we do not assume that in theorem 2.2. Anyway we use the word 'laziness' in a remark inside the proof.***
In discrete time we often assume that $X_t$ is lazy, i.e.\ that at each step, $\Pr(X_{t+1}=X_t)\ge \frac12$, and we will not state this explicitly. In particular, the mixing time in Theorem \ref{thm:short} is for the lazy chain. (Without laziness issues of bipartiteness and near bipartiteness  pop up, which have little theoretical or practical interest; see e.g.\ \cite{tave,PS} and \cite[Remark 1.9]{cf:Basu}). 

The other notions we are interested in are %the $L_2$ distance,
the $L_2$ and uniform
mixing time and the average $L_2$ mixing time. Here are the relevant definitions.
\begin{definition}Let $X_t$, $\pi$ and $P_t(x,y)$ be as above. Then the $L_2$ mixing time, the $L^\infty$ (or uniform) mixing time
and the average $L_2$ mixing time are, respectively,
\begin{align*}
  \mix^{(2)}&\coloneqq \min\{t:\max_x\|\Pr_x(X_t=\cdot )-\pi \|_{2,\pi}\le 1/2\},\\
  \mix^{\mathrm{unif}}&\coloneqq \min \left\{t:\max_{x,y}\left|\frac{\Pr_x(X_{t}=y)}{\pi(y)}-1\right|\le 1/4 \right\},\\
  \mix^{\mathrm{ave}}&\coloneqq\min \left\{t: \sum_{x} \pi(x) \|\Pr_x(X_t=\cdot )-\pi  \|_{2,\pi}^2  \le 1/4 \right\}.
\end{align*}
Here and below,
\[
\|\mu-\pi \|_{2,\pi}^2
\coloneqq\sum_{x}\pi(x) \left(\frac{\mu(x)}{\pi(x)}-1 \right)^2=-1+\sum_{x}\frac{\mu(x)^{2}}{ \pi(x)}.
\]
\end{definition}
The constants $\frac12$ and $\frac14$ do not play an important role and were chosen for convenience. We remark that in the reversible setting the $L^2$ and the $L^\infty$ mixing times satisfy $\mix^{\mathrm{unif}}=2\mix^{(2)}$, while even without reversibility  $\mix^{\mathrm{unif}}\le 2\mix^{(2)}$. See \cite[equation (8.5)]{MT06} and \cite[equation (2.2)]{cf:Spectral} for a proof in continuous time. The proof in discrete time is similar.

In the remainder of this section we show the following.
\begin{theorem}\label{thm:L2mix}
  The average $L_2$ mixing time is robust for reversible Markov chains in continuous time. %This holds  in the continuous time setup and under a certain laziness assumption also in the discrete time setup. The laziness assumption is that both the transition matrices $P$ and $P'$ which are considered satisfy that their holding probabilities are bounded away from $0$ (i.e.\ $\min_x P(x,x) \wedge P'(x,x) \ge  a $ for some constant $a \in (0,1)$). 
\end{theorem}
An immediate corollary is that the (usual, not averaged) $L_2$ mixing time is robust in the transitive setup, under perturbations that preserve transitivity (in the discrete time case assuming the holding probabilities are bounded away from 0). By the remark above, the same holds for the uniform mixing time.
Theorem \ref{thm:L2mix} is not needed for the proofs of our main results. We added it for the sake of completeness. The proof is similar to the one in \cite{cf:Pittet}.% It is meant to demonstrate why our results are surprising.

\begin{proof} %In the discrete time setup, let $P$ be the transition matrix and  $0=\lambda_1 <\lambda_2 \le \cdots \le \lambda_n \le 2 $ be the eigenvalues of $I-P$. In the continuous time setup, l
  Let $\cL$ be the Markov generator, and let $0 =\lambda_1 < \lambda_2 \le \cdots \le \lambda_n$ be the eigenvalues of $-\cL$. Denote $P_t=e^{t\cL}$. %The $L_2$ distance square of $\mu$ from $\pi$ is defined as $\|\mu-\pi \|_{2,\pi}^2=\sum_{x}\pi(x) \left(\frac{\mu(x)}{\pi(x)}-1 \right)^2=-1+\sum_{x}\frac{\mu(x)^{2}}{ \pi(x)}$. Denote by $(X_t)_{t\ge 0}$ the random walk (whether in discrete or in continuous time; in discrete time $t$ below will be an integer).
  Then,
\begin{align}
  \|\Pr_x(X_t=\cdot )-\pi \|_{2,\pi}^2+1
  &=\sum_{y}\frac{\Pr_x(X_{t}=y)^{2}}{ \pi(y)}\nonumber\\
  &\stackrel{(*)}{=}\sum_{y}\frac{\Pr_x(X_{t}=y)\Pr_y(X_{t}=x)}{ \pi(x)}
  =\frac{\Pr_x(X_{2t}=x)}{\pi(x)},  \label{e:L2xx}
\end{align}
where in $(*)$ we used reversibility. 
Hence,
\begin{align}
  \sum_x \pi(x)\|\Pr_x(X_t=\cdot )-\pi \|_{2,\pi}^2 &=-1+\sum_{x}\Pr_x(X_{2t}=x) \nonumber\\
  %=\begin{cases}
  %  \mathrm{Tr}(P^{2t})-1=\sum_{i=2}^n (1-\lambda_i)^{2t}
  %  & \text{in discrete time} \\
      \label{e:avel2}
      &=\mathrm{Tr}(P_{2t})-1=\sum_{i=2}^n \exp(-2 \lambda_i t).
  %  & \text{in continuous time}. \\
  %\end{cases}
\end{align}
%For reversible Markov chains, the average $L_2$ mixing time is defined as
Recalling the definition of the average $L_2$ mixing time we get
\begin{align}
  \mix^{\mathrm{ave}} %  \Inf \left\{t: \sum_{x} \pi(x) \|\Pr_x(X_t=\cdot )-\pi  \|_{2,\pi}^2  \le 1/4 \right\}
  &=%\stackrel{\textrm{\eqref{e:avel2}}}{=}
  \inf \left\{t: \sum_{x}\Pr_x[X_{2t}=x]  \le 5/4 \right\} \nonumber\\
  &=%\begin{cases}
    %\inf \left\{t: \sum_{i=2}^n (1-\lambda_i)^{2t} \le 1/4\right\}
    %& \text{in discrete time} \\
    \inf \left\{t: \sum_{i=2}^n \exp(-2\lambda_it) \le 1/4\right\}.
    %& \text{in continuous time}.
  %\end{cases}
  \label{e:aveL2def}
\end{align}
Using \eqref{e:lailai'}  concludes the proof.% in the continuous time setup. For the discrete time setup use \eqref{e:lailai'} together with the discussion before the proof. 
\end{proof}
\begin{remark}The same calculations can be done in discrete time, leading to  analogues of \eqref{e:avel2} and \eqref{e:aveL2def}: $  \sum_x \pi(x)\|\Pr_x(X_t=\cdot )-\pi \|_{2,\pi}^2 =\sum_{i=2}^n \beta_i^{2t}$ and so
  \[
  \mix^{\mathrm{ave}} = \inf \left\{t \in \mathbb{N} : \sum_{i=2}^n \beta_i^{2t} \le 1/4\right\},
  \]
  where $1=\beta_1>\beta_2 \ge \cdots \ge \beta_n>-1$ are the eigenvalues of the transition matrix $P$ (assuming $P$ is irreducible and aperiodic). This would allow to conclude a similar result in discrete time if it weren't for values of $\beta_i$ close to either 0 or $-1$. Both problems can be resolved by adding laziness, but in the interest of brevity we skip the details.
\end{remark}
%For transitive\footnote{A Markov chain with transition matrix $P$ is said to be \emph{transitive} if for any two states $x,y$ there is a permutation of the state space $\phi$ such that $\phi(x)=y$ and $P(a,b)=P(\phi(a),\phi(b))$ for all $a,b$.}  Markov chains $\Pr_x[X_{2t}=x]  $ is independent of $x$, and so this, together with \eqref{e:L2uniform}, implies that the $L_2$  and the uniform mixing times  are robust under small perturbations which preserve transitivity.

\subsection{Geometric notions} Recall from \S\ref{s:qirob} the definition of $(a,b)$-quasi-isometry for metric spaces, and that when we say that graphs are $(a,b)$-quasi-isometric we are treating them as metric spaces with the graph distance as the metric.%We now define quasi-isometries, the Cheeger constant and the spectral-gap. We state Cheeger's inequality, and give a useful lemma for bounding the Cheeger constant.    
%\begin{definition}
%\label{def: RI}
%Let $G_i\coloneqq (V_i,E_i)$ ($i=1,2$) be two finite graphs. For $u,v \in V_i$, let $d_i(u,v)$ be graph distance (w.r.t.~$G_i$) between $u$ and $v$ (i.e.~the number of edges along the shortest path in $G_i$ between $u$ and $v$). We say that $f:V_1 \to V_2$ is a $(a,b)$ \emph{\textbf{quasi-isometry}} of $G_1$ and $G_2$ if 
%\begin{itemize}
%\item[(1)]
%\[ \forall u,v \in V_1, \quad  \frac{d_1(u,v)-b}{a}\le d_2(f(u),f(v)) \le ad_1(u,v)+b. \]

%\item[(2)]For every $w \in V_2$, there exists some $v \in V_1$ such that $d_2(f(v),w) \le a+b $.
%\end{itemize}
%We say that $G_1 $ and $G_2$ are $(a,b)$ \emph{\textbf{quasi-isometric}} if there exists such $f$ as above.
%\end{definition}

\begin{definition}
\label{def: Cheeger}
Consider a reversible Markov chain  on a finite state space $\Omega $ with transition matrix $P$ (in continuous time, with generator $\cL$) and stationary distribution $\pi$. We define the {\em Cheeger constant} of the chain as  
\begin{equation*}
 \Phi\coloneqq \min_{A: \, 0< \pi(A) \le 1/2}Q(A,A^{c})/\pi(A), \quad \text{where} \quad
\end{equation*}
\[ Q(A,A^{c})\coloneqq \begin{cases}\sum_{x \in A,y \notin A }\pi(x)P(x,y) & \text{in discrete time} \\
\sum_{x \in A,y \notin A }\pi(x)\cL(x,y) & \text{in continuous time}. \\
\end{cases} \]
%is the ergodic flow from $A$ to $A^c$.
\end{definition}
%When $V$ is countable and $P$ is not positively-recurrent $\min_{A: \, 0< \pi(A) \le 1/2}$ is replaced by $\inf_{A: \, 0< \pi(A)<\infty }$.
We will also need a version for a subset of the graph (this is the discrete analogue of Dirichlet boundary conditions). 
\begin{definition}
\label{def: lambdaA}
Let $\Omega$, $P$, $\cL$ and $\pi$ be as above. %Consider a reversible Markov chain with a finite state space $\Omega$,  stationary distribution $\pi$ and with transition matrix $P$ (respectively generator $\cL$).
Let $A \varsubsetneq \Omega $. We define $\Phi(A)\coloneqq  \min_{B \subset A }Q(B,B^{c})/\pi(B) $.

Further, we define $\gl(A)$ to be the smallest eigenvalue of the substochastic matrix obtained by restricting $I-P$ (respectively $-\cL$) to $A$. 
\end{definition}
The following discrete version of Cheeger's inequality under Dirichlet boundary conditions is well known, see e.g. %extension of \eqref{eq: Sinclair} is due to Goel et al.~
\cite[(1.4) and Lemma 2.4]{cf:Spectral}. For every irreducible discrete- or continuous-time reversible chain, and every set $A$ with $\pi(A) \le 1/2$ we have that
\begin{equation}
\label{eq: restrictedcheeger}
\Phi^2(A)/4 \le \gl(A) \le \Phi(A), \qquad \frac{\Phi^2(A)}{4\max_{a \in A}|\cL(a,a)| } \le \gl(A) \le \Phi(A)
\end{equation}
in discrete and continuous time, respectively.

\begin{lemma}\label{lem:conductances}Let $G$ be a finite graph, $v$ a vertex of $G$ and $A_1,\dotsc,A_k$ the components of $G\setminus\{v\}$, i.e.\ of $G$ after removal of the vertex $v$ and all adjoining edges. Let $w_i\in A_i$ be vertices. Then the probability that random walk starting from $v$ hits $\{w_1,\dotsc,w_k\}$ at $w_i$ is proportional to the effective conductance from $v$ to $w_i$.
\end{lemma}
For a gentle introduction to electrical networks see \cite{DS84}.
\begin{proof}
  Denote by $T_{w_i}$ the hitting time of $w_i$ and by $T_W$ the hitting time of the set $\{w_1,\dotsc,w_k\}$. If the walker returns to $v$ before $T_W$ the process begins afresh, so it is enough to consider only the last excursion from $v$. In other words, the probabilities are proportional to the conditioned probabilities $\mathbb{P}_v(T_{w_i}=T_W \,|\,T_W<T_v)$ (we define $T_v$ to be the return time to $v$). Since each $w_i$ is in a different component of $G\setminus\{v\}$, these conditional probabilities are proportional to $\mathbb{P}_v(T_{w_i}<T_v)$. These are proportional to the effective conductance, see \cite[Exercise 2.47]{Lyons}.
\end{proof}
Let $G_{1}\coloneqq (V_{1},E_{1})$ be some graph. Let $G_{2}=(V_{2},E_{2})$ be a graph obtained from $G_{1}$ by `stretching' some of the edges of $G_{1}$ by a factor of at most $K$ (we say that $G_{2}$ is a $K$-\emph{stretch} of $G_{1}$). That is, for some $E \subset E_{1}$ we replace each edge $uv\in E $ by a path of length at most $K$ (whose endpoints are still denoted by $u$ and $v$). Note that $V_{1} \subset V_{2}$. The identity map is a $(K,0)$-quasi-isometry of $G_1$ and $G_2$.
%The following proposition will be useful in what comes.
\begin{lemma}
\label{p: LS}
 There exists a constant $c_d> 0$ (depending only on $d$) such that if $H$ is a simple graph of maximal degree $d$ and
 $G$ is a $K$-stretch of $H$, then 
 \begin{equation}
 \label{eq: LS1}
 \Phi(G) \ge c_d \Phi(H)/K,% \text{ and so }\gl_{2}(G) \ge c_d^{2}\Phi^2(H)/(2K^2),
 \end{equation}
 where $\Phi(G)$ and $\Phi(H)$ are the Cheeger constants of %lazy RW on
 $G$ and $H$, respectively. %An analogous statement holds for the corresponding continuous-time random walks with edge rates 1.  
\end{lemma}
This is well known and easy to see. See e.g.\ \cite[Proposition 2.3]{unifsensitivity} for a proof. 
We finish this section with a simple lemma on stretched trees.
\begin{lemma}\label{l:hit2-h}
Let $T$ be a finite binary tree of height $\ell$, let $f:\{1,\dotsc,\ell\}\to\N$ be non-increasing, and let $G$ be the graph one gets by stretching each edge between levels $h-1$ and $h$ of $T$ to a path of length $f(h)$. Then for every $v$ in level $h$ of $T$ we have
  \[
  \Pr_G(\textrm{random walk starting from $v$ hits the root before hitting level }\ell)\le 2^{-h}.
  \]
%  \tb{[I changed $h$ to $v$ in the above probability]}
\end{lemma}
\begin{proof}The symmetry of the problem allows us to identify all the vertices in each level of $T$ (before stretching). Consider the probability that random walk starting from $v$ hits level $\ell$ before hitting the root. After the identification we have the following.
  \begin{itemize}
  \item Level $\ell$ is just one vertex (which we also denote by $\ell$).
  \item Removing the vertex corresponding to level $h$ (which we also denote by $v$) disconnects the root from $\ell$.
  %\item The random walk until hitting either of the root or $\ell$ is restricted to a finite graph.
    \end{itemize}
  Hence Lemma \ref{lem:conductances} may be used. %But in the graph one gets, removing the vertex corresponding to level $h$ (denote this vertex also by $v$) disconnects the root from infinity, hence Lemma \ref{lem:conductances} may be used.
  Suppressing the dependence on $\ell$, denote the resistances from the root and from $v$ to $\ell$ by $R_1$ and $R_2$, respectively. Then the probability to hit the root before hitting $\ell$ is $R_2/R_1$. %Taking $\ell$ to $\infty$ we get that the probability we are interested in is $R_2/R_1$, which are the resistance from the root and from $v$ to infinity.
  These resistances can be computed directly using parallel-series laws. Indeed, the resistance of $f(i)$ edges in a series is $f(i)$ and the resistance of $2^i$ parallel connections of this kind between $i$ and $i+1$ is $2^{-i}f(i)$. All in all we get
\[
R_1 =\sum_{i=1}^{\ell}\frac{f(i)}{2^i}\ge \sum_{i=1}^\ell \frac{f(i+h)}{2^{i}}
\ge \sum_{i=1}^{\ell-h}\frac{f(i+h)}{2^i}=2^{h}R_2,
\]
where the first inequality follows because $f$ is non-increasing. The assertion of lemma follows. % by dividing by $R_1$ and taking the limit as $\ell \to \infty$.%As the probability in question is $R_2/R_1$ (by Lemma \ref{lem:conductances} and a little algebra), the lemma is proved.
%  First use a network reduction to replace each  path of length  $f(i)$ connecting levels $i-1$ and $i$ by a single edge of edge weight $1/f(i)$, to obtain a weighted binary tree. Then use spherical symmetry to identify all vertices of the same distance from the root, to obtain a birth and death chain on $\mathbb{Z}_+$ in which state $i$ corresponds to level $i$ of the weighted binary tree.  The edge weight between states $i-1$ and $i$ in this birth and death chain is given by $1/R_{i} :=  2^i/f(i)$. We get that the probability of reaching the root is equal to $\frac{\sum_{i \ge h}R_i}{\sum_{i \ge 0}R_i}=\frac{\sum_{i=h}^{\infty}f(i)2^{-i}}{\sum_{i=0}^{\infty}f(i)2^{-i}} \le \frac{\sum_{i=h}^{\infty}2^{-i}}{\sum_{i=0}^{\infty}2^{-i}}=2^{-h}$, where we have used the fact that $f$ is non-increasing.
\end{proof}

\subsection{A tail estimate for hitting times}
\label{s:LDandHT}
%We shall need the following large deviation result. %***check that we still %need it.*** %Denote the Lebesgue measure by $\mathrm{Leb}$. 
%\begin{lemma}[\cite{LD} Theorem 3.4]
%\label{l:LDrev}
%Consider a reversible continuous time Markov chain with a finite state space %$\Omega$,  stationary distribution $\pi$ and spectral gap $\lambda_{2}$. %Let $A \subset \Omega$. Let $N(A,t)$ be the Lebesgue measure of $\{s \in %[0,t]  :X_s \in A \}$. Then
%\begin{equation}
%\forall \, t,\eps \ge 0, \quad \Pr_{\mu}[N(A,t)<t(\pi(A)-\eps ) ] \le C %\|\mu \|_{2,\pi}  \exp \left(-c\lambda _{2}t \eps^2 \right),
%\end{equation}
%for every distribution $\mu$ on $\Omega$, where $c,C>0$ are absolute constants %(independent of the chain)
%and where $ \|\mu \|_{2,\pi}^{2}\coloneqq \sum_{x \in \Omega}\pi(x) (\mu(x)/\pi(x))^2$. %\end{lemma}
%
%The Poincar\'e (spectral gap) inequality asserts that when time is scaled %according to the inverse of the spectral gap, the $L_2$ distance from stationarity %of every distribution decays exponentially in the number of (scaled) time %units.
%\begin{lemma}
%\label{lem: Poincare}
%Let $(\Omega,P,\pi)$ be a finite lazy  irreducible reversible Markov chain %with spectral gap $\gl$.
%Let
%$\mu $ be a distribution on $\Omega$. Then
%\begin{equation}
%\label{eq: L2contraction}
% \|\Pr_\mu^t-\pi \|_{2,\pi} \le  e^{-\gl t}
%\|\mu-\pi \|_{2,\pi}, \text{ for all }t \ge 0.
%\end{equation}
%\end{lemma}  

Recall that the \emph{hitting time} of a set $D$ is defined as $T_D\coloneqq \inf \{t\ge 0:X_t \in D\} $. Denote $\pi$ conditioned on $A$ by $\pi_A$, i.e.\ $\pi_A(a)= \mathbbm{1}\{a \in A \}\frac{\pi(a)}{\pi(A)} $. Using the spectral decomposition of $P_A$ (the restriction of the transition matrix $P$ to the set $A$) with respect to the inner product $\langle f,g \rangle_{\pi_A}\coloneqq \sum_{a}\pi_A(a)f(a)g(a) $, we get (see e.g.~\cite[Lemma 3.8]{cf:Basu} or \cite[Ch.\ 3]{cf:Aldous}), 
\begin{equation}
\label{eq: exitprob}
\min_{x,y \in A} \frac{\pi(x)}{|A|\pi(y)} \sum_{a \in A}  \Pr_{a }[T_{ A^{c}} > t]\le \Pr_{\pi_A}[T_{ A^{c}} > t] \le   e^{- \gl(A)t}
\end{equation}
in discrete or continuous time. 

 \section{Proof of Theorems \ref{thm:short} and \ref{thm:weighted}}
 \label{s:proofthm3}

Throughout this section we consider  the interchange process on a graph $G$ in continuous time in which all edges ring at rate 1 (Theorems \ref{thm:short} and \ref{thm:weighted} are formulated in discrete time, but translating the mixing time from continuous time to discrete time is simple and we explain this for Theorem \ref{thm:short} at the end of its proof, the explanation there holds for Theorem \ref{thm:weighted} equally). Since the claims of both theorems are asymptotic, we may and will assume that $n$ is sufficiently large.

 Let us start the proof by recalling elements of the construction already discussed in the proof sketch in \S\ref{s:sketch} and in other places in the introduction. We need to find a set of transpositions $S_n\subset\sym_n$ such that $\mix(\Cay(\sym_n,S_n))$ is small compared to either $\mix(\Cay(\sym_n,S_n'))$ for a second set of transpositions $S_n'$ such that $S_n\subseteq S_n'\subseteq S_n^3$ (in Theorem \ref{thm:short}) or to $\mix(\Cay(\sym_n,S_n,W_n))$ for some weights $W_n$ (in Theorem \ref{thm:weighted}). 
 We describe our set of transpositions using a graph $G$ on $n$ vertices, whose edges are the transpositions.  The construction has two parameters, $u\in\N$ and $\eps\in(0,\frac12)$ (both will be chosen later). We designate $u$ parts of $G$ and call them $H_1,\dotsc,H_u$ (we will use $H_i$ to denote both a subset of $[n]$ and the induced subgraph, and we will now describe them as graphs, thus describing also a part of $G$). The $H_i$ are constructed inductively as follows. The induction base, $H_1$ is a binary tree of depth $u$ whose edges have been replaced by paths of length $2^{u}$. To define $H_{i+1}$ given $H_{i}$ we label, in each vertex of each of the trees used to construct $H_{i}$ one child as `left' and the other as `right'. We denote, for each leaf $v$ of $H_{i}$, the number of left children on the path from the root to $v$ by $g(v)$. Recall the definition of the bad leaves $B_{i}$,
\begin{equation}\label{eq:defBi}
B_i\coloneqq \Big\{x\textrm{ leaf of }H_i:g(x)>\left(\sdfrac12+\eps\right)s_i\Big\}
\end{equation}
from \eqref{eq:defBi sketch}. We define $H_{i+1}$ as a forest of $|B_{i}|$ binary trees of depth $s_{i+1}$, with each edge replaced by a path of length $\ell_{i+1}$, with
 \[
 s_{i}\coloneqq  4^{i-1}u\qquad \ell_{i}\coloneqq  2^{u+1-i}
 \]
 and each tree rooted at a point of $B_{i}$ (so $H_{i+1}\cap H_{i}=B_{i}$ as sets). This terminates the description of the $H_i$. All this, we remind, was already discussed in \S\ref{s:sketch} with some additional explanations and motivation (and a figure depicting the gadget $\bigcup H_i$ on page \pageref{fig:gadget}).

We now claim that, uniformly in $\epsilon$,
\[
2^{4^{u-1}u}\le \bigg|\bigcup_{i=1}^u H_i\bigg|\le 2^{4^{u}u}.
\]
Indeed, the first inequality is clear because $H_u$ has at least one root (since $\eps<\frac 12$) and the second inequality comes from %\tb{[I changed $2^{u+1-i}$ below to $2^{u+2-i}$]}
\[
|H_i|\le \ell_i2^{s_i+1}|B_{i-1}|
\le 2^{u+2-i+4^{i-1}u}\prod_{j=1}^{i-1}2^{4^{j-1}u}
\le 2^{u+2-i+4^{i-1}u(1+1/3)}\le 2^{3\cdot 4^{i-1}u}
\]
which can be summed readily to give
\begin{equation}
  \label{eq:|Hi|}
  \sum_{j=1}^i|H_j|\le 2^{4^iu},
\end{equation}
and in particular the case $i=u$ is what we need. Hence we may choose some $u=u_n$ such that $u4^u\asymp\log n$ (in particular $u\asymp\log\log n$) such that
\begin{equation}\label{eq:gadgetsmall}
\sum_{i=1}^u|H_i|\le n^{1/4},
\end{equation}
regardless of $\eps$ (we need here $n\ge 65,\!536$ to have $2^{4^uu}\le n^{1/4}$ for $u=1$). %\footnote{\tb{It will be important in the proof of Theorem \ref{thm:weighted}, where we will take $\eps=\eps_n=o(1)$, that we still have that \eqref{eq:gadgetsmall} holds for some $u$ such that $u\asymp\log\log n$), which indeed follows from the above analysis.}}
Fix such a $u$ for the rest of the proof.

The subgraph $\bigcup H_i$ is the `gadget', and the rest of the graph $G$ is a complete graph on a set of vertices $K$. The gadget connects to the complete graph via the good leaves of the $H_i$ (and all the leaves of the last one, $H_u$) so we define $K$ to also include those vertices. Thus we define
\[
K\coloneqq  \bigg([n]\setminus\bigcup_{i=1}^uH_i\bigg)
\cup\bigg(\bigcup_{i=1}^{u-1}\left(\{\textrm{leaves of }H_i\}\setminus B_i\right)\bigg)
\cup\{\textrm{leaves of }H_u\}.
\]
Let the edges of $G$ be all the edges of all the $H_i$ union with a complete graph on $K$ i.e.
\[
E(G)\coloneqq \bigg(\bigcup_{i=1}^uE(H_i)\bigg)\cup \{\{x,y\}:x,y\in K\}.
\]
This finishes the construction of $G$ (except for the choice of $\eps$), and hence of $S_n$. We delay the definitions of $S_n'$ and $W_n$ to \S\ref{s:pertoverview}.

% \textcolor{blue}{Above we fixed some  $u=u_n$ such that $u4^u\asymp\log n$. The choice of $u_n$ did not depend on the parameter $\varepsilon$. We emphasize that there exists some $N_0 \in \N$ (independent of $\varepsilon$) such that the above construction can be carried out for all $n \ge N_0$  for all  $\varepsilon \in (0,1/2)$. We shall suppress the dependence of $u$ and possibly also $\varepsilon$ on $n$ from the notation. We shall also suppress the dependence of the graph $G$ on $(n,u,\varepsilon)$ from the notation.
Below, when we want to emphasize the dependence on $u$ and $\varepsilon$ we will write $G_n(u,\varepsilon)$ for $G$ and $S_n(u,\varepsilon)$ for $S_n$. We will also denote $G_n(\varepsilon)\coloneqq G_n(u_n,\varepsilon)$ and $S_n(\varepsilon)\coloneqq S_n(u_n,\varepsilon)$ (recall that $u_n$ is the value we fixed above such that $u_n4^{u_n}\asymp\log n$).
\subsection{An upper bound for the time to exit the gadget}
\leavevmode

\medskip
\noindent Throughout the proofs we will pick the parameter $\varepsilon$ so that $\eps>u^{-1/3}$.
%The SRW on each $H_i$ viewed at times it is in $V(H_i')$, after omitting consecutive visits to the same site, evolves like SRW on $H_i'$. Hence it has a drift towards its leaf set. The typical time between two such steps is $\Theta(\ell_i^2)$. The following lemma follows from this observation. 
\begin{lemma}
\label{l:LLilem}
  The expected  exit time from $H_i$, starting from a worst initial state in $H_i$ (i.e.\ the one maximizing this expectation), denoted by $L_i$, satisfies (uniformly in $i$)
\begin{equation}
\label{e:KiHi}
L_i \asymp\ell_i^2s_i = u4^u=:L. 
\end{equation}
\end{lemma}
As this lemma is standard we only sketch its proof.
\begin{proof}[Proof sketch] Examine the random walk $X$ on $H_i$ and let $\sigma_0,\sigma_1,\dotsc$ be the times when it reaches a vertex of degree 3 (we require also $X_{\sigma_{i+1}}\ne X_{\sigma_i}$). Between $\sigma_i$ and $\sigma_{i+1}$ the walk is in a part of the graph which is simply 3 paths of length $\ell_i$. By symmetry it reaches each of the 3 ends of these lines with equal probability. Hence $X_{\sigma_i}$ is identical to a random walk on a binary tree of depth $s_i$. The distance of random walk on a binary tree from the root has the same distribution as a random walk on $\N$ with a drift towards infinity, and hence a simple calculation shows that the expected exit time is $Cs_i$. To get back to random walk on $H_i$ we note that, even if we condition on $X_{\sigma_0},X_{\sigma_1},\dotsc$ then the local symmetry says that the times $\sigma_{i+1}-\sigma_i$ are independent of $X_{\sigma_i}$ and of one another. For each $i$ we have $\ex(\sigma_{i+1}-\sigma_i)\asymp \ell_i^2$, because this is the same as the exit time from the interval $\{0,\dotsc,\ell_i\}$, where the walk exits 0 at rate 3 (and the other vertices at rate 2), again by the symmetry. %\tb{[Why triple? Isn't it more like rate 3 at 0 compared to rate 2 elsewhere?]}
\end{proof}
Below we employ the notation $L=u4^u$ from the above lemma.
Let $ T_{K}:=\inf \{t:X_t \in K \}$ be the hitting time of the complete graph $K$. Recall that we have fixed a choice of  $u=u_n$  satisfying that $u4^u\asymp\log n$.    %The following proposition is our main will be key in bounding the mixing time of the interchange process. %It is also the hardest component in the proof of theorems \ref{thm:short} and \ref{thm:weighted}.
\begin{proposition}
\label{p:tauHi}
%Let $L$ be from \eqref{e:KiHi}, and let the parameter $\eps$ from the construction be arbitrary. Then f
There exist some constants $C$ and $c'$ such that for every $\varepsilon$ and $n$ we have that the graph $G_n(\varepsilon)$ satisfies  for all $i\in[u]$ that
\begin{equation}
\label{e:escape4b}
\sum_{v \in H_{i} }\Pr_v\Big[T_{K}> \frac{CL}{\eps^4} \Big] \le C \frac{1}{|H_{i}|^{c'\eps^2}}.
\end{equation}
Consequently, if $E$ is the event that for all $i \in [u]$ all particles whose initial location is in $ H_i$ hit the complete graph $K$ before time $CL/\eps^4$, then $\lim_{n\to\infty} \Pr(E)=1$, uniformly in $\eps$.
%for any $(\varepsilon_n)_{n=N_0}^{\infty}$ such that $\varepsilon_n \gg (\log \log n)^{-1/2}$ and  $u_{n}^{-3}<\varepsilon_n <1/2$ for all $n$, the interchange process with edge rate 1 on the graph $G(n,u_{n},\varepsilon_n)$ satisfies a.a.s.\ as $n\to\infty$ that all particles hit the complete graph $K$ by time $CL/\eps_{n}^4$.
%% *** I think this is not needed.***
%% There exists some $m=m_{\eps}$ such that
%% \begin{itemize}
%% \item[(a)] When $\eps$ is constant so is $m$.
%% \item[(b)] For any diverging $g(n)$ we can pick $\eps=\eps_n=o(1)$ so that $m_{\eps_n} \le g(n)$ for all $n$.
%% \item[(c)] For some absolute constant $C>0$ we have for all $i \in [n]$ that 
%% \begin{equation}
%% \label{e:escape4b}
%% \sum_{v \in V( H_{i}) }\Pr_v[\tau> C4^{2m}L] \le \frac{1}{|H_{i}|^{2}}. 
%% \end{equation}
%% \end{itemize}
%% Consequently, the interchange process on $G_n$ satisfies that w.h.p.\ *** if this means $1-o(1)$ then change to a.a.s. and define, otherwise explain*** all particles cross an edge of the clique $K$ by time $C8^mL$.
%% *** until here ***
\end{proposition}
We recall our standing assumptions that $n$ is sufficiently large and that $\eps>u_n^{-1/3}$ (in particular, `uniformly in $\eps$' above means `uniformly in $\eps\in(u^{1/3},\frac12)$'). %that the acronym a.a.s.\ stands for asymptotically almost surely and means that the probability of the event goes to 0 as $n\to\infty$.
Let us remark that the $\eps^4$ term is not optimal, but this is not a priority for us.
%% We shall only consider $\eps=\eps_n$ such that $4^{2m}=o(n)$.
%% We will show that the order of the mixing time of the interchange process before the perturbation is\footnote{When $\eps$ is constant $4^{2m}L \asymp L $. Hence this upper bound matched the  lower bound $\Omega(L)$, corresponding to the time it takes a single particle  to exit either one of the sets $V(H_i)$, coming from Lemma \ref{l:LLilem}.} $O(4^{2m}L)$, whereas after the perturbation the mixing time of a single particle is $\Theta(\sum_{i \in [n]} L_i)=\Theta(nL)$ (and thus that of the interchange process after the perturbation is $\Omega(nL)$; with more care one can verify it is actually $\Theta(nL)$, but this will not be used). 
\begin{proof}[Proof of Proposition \ref{p:tauHi}]
  The assertion of the last sentence of the proposition follows from \eqref{e:escape4b} by a union bound over the particles (recall that in the interchange process each particle is performing a random walk). We also need here our assumption that $\eps > u^{-1/3}$, as it gives  $\sum_{i=1}^u 1/|H_i|^{c\eps^2}=o(1) $, since $|H_i| \ge 2^{s_i} = 2^{4^{i-1}u}$. %and that for all $t\gg  | V(K)|^{-1}$ (in particular for $t= C4^{2m}L$) \[\sum_{v \in V(K) \setminus \bigcup_{i \in [n]}V(H_i) }\Pr_v[\tau>t] \le | V(K)|e^{-(| V(K)|-1)t} =o(1) .\] 

Thus we need to verify \eqref{e:escape4b}. Let $m=m(\eps)\ge 1$ be some integer parameter to be fixed later. % (for the impatient we can say that we will have $m\asymp\log\frac 1\eps$). %Let $\eps$ be as in \eqref{e:BDi}. Let $m=m_{\eps} \in \N $ to be determine later (it will be picked so that (a) and (b) above hold).
Let 
\[
\rho_i(t)\coloneqq \sum_{v \in H_{i} }\Pr_v[\min\{T_{[n]\setminus W_i},T_{K}\}>t],
%\qquad J_i\coloneqq [n]\setminus W_i
\qquad W_i=W_i(m)\coloneqq \bigcup_{\mathclap{j \in[i-m,i+m] \cap [u] }} H_j,
\]
where $T_{[n]\setminus W_i}$ is the hitting time of $[n]\setminus W_i$ (or the exit time of $W_i$, if you prefer). %Note that the union defining $W_i$ might contain irrelevant $j$ i.e.\ $j\not\in[u]$. These do not contribute, e.g.\ $W_1$ is in fact $\bigcup_{j=1}^{1+m}H_j$.
The proof of Proposition \ref{p:tauHi} is concluded by combining the following two lemmas.
%Proposition \ref{p:tauHi} follows immediately.
Indeed, let $m$ be the minimal value which satisfies the requirement of Lemma \ref{lem:TJibeforetau}, so $4^m\asymp \eps^{-2}$. We use the same value of $m$ in Lemma \ref{lem:rhoit} and get that for $t>CL/\eps^4$ we have \eqref{e:escape}. Combining this with (\ref{e:escape2}) %the conclusion of Lemma \ref{lem:TJibeforetau}
gives the proposition. %\qed
%\phantom\qedhere
\end{proof}
\begin{lemma}
  \label{lem:rhoit}
  For all $i\in[u]$ and all $t\ge C_116^mL$ for some $C_1$ sufficiently large,
\begin{equation}
\label{e:escape}
% \max_{i \in [u]}  \inf \left\{t:\rho_i(t)\le \frac{1}{8|H_{i}|^{2}} \right\} \lesssim 4^{2m}  L .
\rho_i(t)\le \frac{1}{|H_i|^2}.
\end{equation}
\end{lemma}
\noindent (As usual, $C_1$ is an absolute constant. In particular it depends on neither $i$ nor $t$.)
\begin{lemma}
\label{lem:TJibeforetau}
%We can pick   $m=m_{\eps} \in \N$ so that  (a) and (b) above hold, and so that (before the perturbation)
There exist absolute constants $C,c>0$ such that for all $\varepsilon \in (0,1/2)$ if $4^m\eps^2 \ge C$  then for all $n \ge N_0$ the graph $G(n,u,\varepsilon)$ satisfies  for every $i\in[u]$ that
\begin{equation}
\label{e:escape2}
\sum_{v \in H_{i} }\Pr_v[T_{[n]\setminus W_i} <T_{K}] \le \frac{C}{|H_{i}|^{c\eps^2}}.
\end{equation}

\end{lemma}
\begin{proof}[Proof of Lemma \ref{lem:rhoit}]
  Let $\cM$ be the restriction of the Markov generator $\cL$ to $W_i \setminus K $ (i.e.\ this is the generator of the chain killed upon exiting $W_i \setminus K $). Let $\lambda$ be the smallest eigenvalue of $-\cM$. It will be convenient to extend the definition $\ell_{i} =2^{u+1-i}$ also to negative $i$. %We employ the convention $\ell_{a}\coloneqq \ell_1$ for $a<1$.
  We now claim that
\begin{equation}
\label{e:Dev}
 \lambda \gtrsim \ell_{i-m}^{-2}.
\end{equation}
To see this, let $W$ be an arbitrary connected component of $W_i$. We first apply Lemma \ref{p: LS} to $W$. Since it is a piece of an infinite binary tree with edges stretched to various extents, but not more than $\ell_{i-m}$, and since the infinite tree has positive Cheeger constant, we get that the Cheeger constant of $W$ is at least $\ell_{i-m}^{-1}$. Applying Cheeger's inequality \eqref{eq: restrictedcheeger} to $W\setminus K$ embedded in an infinite, stretched tree %to the binary tree. Since the binary tree has positive Cheeger constant, it also has a positive spectral gap. To move from the infinite binary tree to $W_i$ (which is some subset of the tree with the edges stretched to various extents, but no more than $\ell_{i-m}$), we apply  This
shows \eqref{e:Dev}.

Using this we get that
%We use the fairly `cheap' bound \eqref{eq: exitprob}, given   in terms of the smallest    eigenvalue $ \lambda(W_i)$ (which we bound from below in \eqref{e:Dev}) of $-\cL_{W_{i}}$ where for a set $A$, $\cL_{A}(a,b)\coloneqq \cL(a,b)\mathbf{1}\{a,b \in A \}$ is the restriction of $\cL$ to $A$, and where $\cL$ is the Markov generator of the walk ($\cL_{A}$ is the generator of the walk killed upon exiting $A$): 
\begin{equation}
\label{e:cheap1}
\begin{split}
  \rho_i(t) & \stackrel{\textrm{\eqref{eq: exitprob}}}{\le}
  |W_i| \exp(- \lambda t   )
  \stackrel{\text{(\ref{e:Dev},\ref{eq:|Hi|})}}{\le}
  2^{4^{i+m}u} \exp(-c\ell_{i-m}^{-2}t   )  \\
  &\;\,\stackrel{(*)}{\le} 2^{4^{i+m}u} \exp(-c 4^{i-u-m-1}\cdot C_116^m\cdot u4^u   )
  \stackrel{\textrm{\eqref{eq:|Hi|}}}{\le}
  \frac{1}{|H_{i}|^{2}},  
\end{split}
\end{equation}
where the inequality marked $(*)$ follows from the definitions of $\ell_i$ and $L$ (recall that $L=u4^u$) and from the bound on $t$ in the statement of the lemma. In the last inequality we also use that $C_1$ is sufficiently large.
%with the last inequality holding for $t=C'L4^{2m} \gtrsim n4^n4^{2m}$ for some sufficiently large $C'>0$. In the second inequality of \eqref{e:cheap1}  we used the estimate
\end{proof}
\begin{proof}[Proof of Lemma \ref{lem:TJibeforetau}]
We divide the event $T_{[n]\setminus W_i}<T_{K}$ into two cases: that the random walk hits $H_{i+m+1}$ before hitting $K$, and that it hits $H_{i-m-1}$ before hitting $K$. Denote these two events by $\cU$ and $\cD$ respectively (notice that if $i\ge u-m$ then $\cU$ is empty and if $i\le m+1$ then $\cD$ is empty). The letters $\cU$ and $\cD$ stand for `up' and `down', with the orientation being as in Figure \ref{fig:gadget} (page \pageref{fig:gadget}).

We first handle $\cU$.
% Let $T$ be a binary tree rooted at $\rho$. Let $j=\sum_{\ell=i+1 }^{i+m}s_{\ell}$. %Denote the $j$th level of $T$ by $\cL_j$. Let $(Y_t)$ be SRW on $T$. Denote %the hitting of a set $B$ by $(Y_t)$ by $\tau_B$. Let $M:=\frac{|B_{i+m}|}{|B_{i}|}. %$ By symmetry and an obvious coupling argument between the SRW $(X_t)$ on %$G$ and the SRW $(Y_t)$ on $T$, in which the event $\cU$ (for $(X_t)$) implies %$\tau_A=\tau_{ \cL_j}$ (for $(Y_t)$), we have for all $x \in H_i$ that
%\[\mathbb{P}_x[\cU] \le \max_{A \subset \cL_j \colon |A|=M } \mathbb{ P}_{\rho}[\tau_A=\tau_{ %\cL_j}]=M/|\cL_j| \] \[=\prod_{\ell=i+1}^{i+m} \Pr\Big(\textrm{Bin}(s_{\ell},\tfrac %12)>\Big(\tfrac 12+\eps\Big)s_{\ell}\Big) \le  \exp \left(-c\eps^2\sum_{\ell=i+1}^{i+m} %s_{\ell}\right). \]
 For $\cU$ to happen there  must be some time $\sigma<T_K$ such that $X_{\sigma}\in B_{i+m-1} \subset H_{i+m}$ and, further, the walker is contained in $H_{i+m}$ between time $\sigma$ and the first hitting time to the set of leaves of $H_{i+m}$ which (on the event $\cU$) occurs at $B_{i+m}$. Assume such a $\sigma$ exists and examine the walker between $\sigma$ and $T_{B_{i+m}}$. The walker is not simple (because being after $\sigma$ conditions it to not return to the roots of $H_{i+m}$) but this is not important for us. The symmetry of the tree implies that at the first time after $\sigma$ that the walker visits a leaf of $H_{i+m}$, the difference between the number of left and right turns along the path the walker takes is distributed like a sum of i.i.d.\ $\pm 1$ variables (giving equal probability to each value). In particular, the probability that the target leaf is in $B_{i+m}$ is
\[
\Pr\Big(\textrm{Bin}(s_{i+m},\tfrac 12)>\big(\tfrac 12+\eps\big)s_{i+m}\Big)\le \exp(-c\eps^2s_{i+m}).
\]
Our assumption that a time $\sigma$ exists only reduces the probability further so we get, for every $v\in H_i$, that $\Pr_v(\cU)\le \exp(-c\eps^2s_{i+m})$. Summing over $v$ gives
\[
\sum_{v\in H_i}\Pr_v(\cU)\le |H_i|\exp(-c\eps^2s_{i+m})
\stackrel{\textrm{\eqref{eq:|Hi|}}}{\le} 2^{4^iu}\exp(-c\eps^24^{i+m}u),
\]
and we see that if $m$ satisfies $4^m\ge 2/c\eps^2$ then this sum is smaller than, say, $1/|H_i|^2$. Require $m$ to satisfy that, but do not fix its value yet (there will be a similar requirement below). %Assume also $m\ge 1$.

We move to the estimate of $\cD$. We use Lemma \ref{l:hit2-h} and get that for any $v$ in level $h$ of $H_i$ (we are counting levels before stretching here) or in the path between level $h$ and $h+1$, we have
\[
\Pr_v(\cD)\le  2^{-h-\sum_{j=i-m}^{i-1}s_{j}}.
\]
(Note that Lemma \ref{l:hit2-h} measures a larger event. Indeed, $\cD$ is the event to hit the root of our tree before hitting level $i+m$ or $K$, so it is smaller than the event to hit the root before $i+m$, which is what is measured by Lemma \ref{l:hit2-h}.) %This only reduces the probability further.)
The number of vertices at level $h$, or in a path between level $h$ and $h+1$, is $|B_{i-1}|\cdot 2^h\cdot \ell_i$ so we get
\begin{equation}\label{eq:sum PvD}
  \sum_{v\in H_i}\Pr_v(\cD)\le\sum_{h=0}^{s_i-1}|B_{i-1}|2^{h}\ell_i\cdot   2^{-h-\sum_{j=i-m}^{i-1}s_{j}}=s_i\ell_i\frac{|B_{i-1}|}{  2^{\sum_{j=i-m}^{i-1}s_{j}}}.
\end{equation}
Denote $p_j:=\Pr\big(\textrm{Bin}(s_{j},\tfrac 12)>(\tfrac 12+\eps)s_{j}\big)$. Then $|B_{j}| \le |B_{j-1}|2^{s_{j}}p_{j}$ and further $p_j \le\exp \left(-c\eps^2 s_{j}\right) $.  Iterating this gives %and using $|B_j| \le 2^{\sum_{\ell=1}^{j}s_{\ell}} \le 2^{2 s_j}$ we get that 
\[|B_{i-1}| \le |B_{i-m-1}|2^{\sum_{j=i-m}^{i-1}s_{j}}\prod_{j=1}^{i-1}p_j \le |B_{i-m-1}|2^{\sum_{j=i-m}^{i-1}s_{j}}\exp \left(-c\eps^2 s_{i-1}\right).
\]
Substituting this in \eqref{eq:sum PvD} gives
\[
\sum_{v\in H_i}\Pr_v(\cD)\le s_i\ell_i|B_{i-m-1}|\exp(-c\eps^2s_{i-1})
\stackrel{\mathclap{\textrm{\eqref{eq:|Hi|}}}}{\le}
s_i\ell_i2^{4^{i-m-1}u}\exp(-c\eps^2s_{i-1}).
\]
We see that taking $m$ so that $4^{m}\eps^2 $ is sufficiently large makes the term $2^{4^{i-m-1}u}=2^{s_{i-1}4^{-m+1}}$ negligible compared to the exponential (recall that $s_i=4^{i-1}u$). This is the last requirement from $m$ and we may fix its value. Further, our standing assumption that $\eps>u^{-1/3}$ means that the $s_i\ell_i$ terms are also negligible with respect to $\exp(-c\eps^2s_{i-1})$. Hence %\tb{[I changed below $c$ to $c'$ since the r.h.s. of the last display is not $\lesssim \exp(-c\eps^2s_{i-1})$.]}
\[
\sum_{v\in H_i}\Pr_v(\cD)\lesssim \exp(-c'\eps^2s_{i-1}),
\]
%% yields
%% \[
%% \sum_{v\in H_i}\Pr_v(\cD)\le s_i\ell_i2^{2s_{i-m-1}}\exp \left(-c\eps^2 s_{i-1}\right)\le s_i\ell_i \exp \left(- \frac 12 c\eps^2 s_{i-1}\right) \le|H_i|^{-c'\eps^2}
%% \]
%% for some $c'>0$.  We now explain the last inequality in the previous display. To get rid of the $\ell_i$ term we recall that $\cD$ is empty unless $i>m+1$. Hence we may write
%% \[
%% |H_i|^{-2c'\eps^2}=|H_i|^{-c'\eps^2}\cdot|H_i|^{-c'\eps^2}
%% %\stackrel{\textrm{\eqref{eq:|Hi|}}}{\le}
%% \le |H_i|^{-c'\eps^2}2^{-4^{i-1}u\cdot c'\eps^2}
%% \]
%% and if $4^m\eps^2$ is sufficiently large, the last term may cancel $s_{i}\ell_i\le u 2^{u+i}$.
as needed. The lemma is thus proved, and so is Proposition \ref{p:tauHi}.
\end{proof}
Having established in Proposition \ref{p:tauHi} that the walker hits $K$, we now show that it remains there for a considerable amount of time.
\begin{lemma}
\label{l:NKt}
%Let $x$ be a vertex in the clique $K=(V(K),E(K))$. Let $\mu$ be the uniform distribution on $V(K) \setminus \{x\}$.
Let $t\coloneqq C_1 L/\eps^4$ for some $C_1$ sufficiently large. For every $x\in[n]$ let $N(x)$ be the amount of time a walker starting from $x$ spends in $K$ up to time $t$. Then
\[
\Pr\left[\exists x\in[n]\textrm{ s.t. }N(x)<\tfrac 23 t\right]\to 0
\]
as $n\to\infty$, uniformly in $\eps>u^{-1/3}$ (but not necessarily in $C_1$).
%% , with $L$ and $m$ as above. Consider the random walk on $G_n=(V_{n},E_n)$ with initial distribution $\mu_{x}$, in which for $xy \in E_n$ the jump rate from $x$ to $y$ is $1$. Let $N(K,t)$ be the amount of time the walk spends in $V(K)$ by time $t$. Then if $C$ is sufficiently large
%% \[\max_{x \in V(K)} \Pr_{\mu_{x}}\left[N(K,t) \le \left(\frac{2}{3}-\frac{1}{100}\right)t \right] = o(1/|V_n|). \]
\end{lemma}
\begin{proof}
Let $q=C_2L/\eps^4$ where $C_2$ is the constant from Proposition \ref{p:tauHi}, denoted there by $C$. Apply Proposition \ref{p:tauHi} after some arbitrary time $s$. We get that during the interval $[s,s+q]$, the probability that all particles in $[n]\setminus K$ hit $K$ is at least
\[
  1-C\sum_{i=1}^u \frac{1}{|H_i|^{c\eps^2}}>1-C2^{-u/4}.
\]
Hence for any fixed value of $C_1$, we can apply this for $s=0,q,2q,\dotsc,q(\lfloor t/q\rfloor-1)$ (just a constant number of times, in fact $\lfloor C_1/C_2\rfloor$) and get that with probability going to 1, all events happen simultaneously. In other words, no particle spent more than $2q$ consecutive time units in any visit of $[n]\setminus K$.

Let us now bound the number of possible visits. We will show that a.a.s.\ as $n\to\infty$ no particle makes more than one visit to $[n] \setminus K$ by time $t$ after reaching $K$ for the first time.  For this purpose denote  $\partial K$ to be all points of $K$ with a neighbour in  $[n]\setminus K$ (namely, leaves of  $H_i$ which are not in $B_i$ for $i<u$ and all leaves of $H_u$).

Suppose a particle is at time $0$ at some $x\in K$. Let us first bound the number of jumps it does up to time $t$. Since the degrees of our graph are all bounded by $|K|$, this number is stochastically dominated by an appropriate Poisson variable, and in particular the probability that the particle performed more than $2t|K|$ jumps is $o(1/n)$. Adding the restriction that the jump would be to a vertex of $K$ only reduces the number further, so we get the same bound for the number of jumps to vertices of $K$.

Among the first  $2t|K|$ jumps to vertices of $K$, the number of jumps to $\partial K$ is stochastically dominated by  $\mathrm{Bin}(2t|K|,\frac{|\partial K|}{|K|-1}) $. Hence the probability that more than $4t|\partial K|$ of them are to $\partial K$ is $o(1/n)$ (where we used that $|\partial K|  \gtrsim n^c $, which follows from our choice of $u$).

Examine now the first  $4t|\partial K|+2$ visits to $\partial K$ (not necessarily up to time $t$, all of them). The probability that at least two of the following jumps were away from $K$ is at most $(4t|\partial K|+2)^{2}/|K|^{2}=o(1/n)$, where we used the fact that by   \eqref{eq:gadgetsmall} $|\partial K|\le n^{1/4}$ and $|K| \ge n-n^{1/4}$, as well as $t \lesssim \log^2 n$ (recall that $L\asymp\log n$ and $\eps>u^{-1/3}\asymp(\log\log n)^{-1/3}$). In the case that indeed no more than 1 of these jumps went to $[n]\setminus K$ we get that the first $4t|\partial K|+2$ visits to $\partial K$ include all the visits to $\partial K$ up to time $t$: no more than $4t|\partial K|$ visits from $K$ and no more than 2 visits from $[n]\setminus K$ (the first hitting of $K$ and the first return to $K$).

Combining everything together, we see that after first reaching $K$ (which a.a.s.\ all particles do by time $q$) a.a.s.\ all particles leave $K$ at most once by time $t$ and during such excursion they each spend at most $2q$ time units away from $K$. Taking $C_1$ to be large enough in terms of $C_2$ concludes the proof.   
%
%This, of course, works starting from any stopping time, not just from \tb{time} %0, and we see that the probability that a walker visits $[n]\setminus K$ %more than 3 times is bounded by $t^4n^{-4/3}=n^{-4/3+o(1)}$ (we bound here %$t=C_1L/\eps^4$ by $L\lesssim \log n$ and $\eps>u^{-1/3}\asymp (\log \log %n)^{-1/3}$). These can be summed over all particles and we see that
%\[
%\Pr[\exists x:x\textrm{ spends more than }6q\textrm{ time in }[n]\setminus %K]\to 0.
%\]
\end{proof}
\subsection{The coupling}
\label{s:coupling}
Denote the transposition $(x,y)$ by $\tau_{xy}$.
 Consider two initial configurations $\sigma$ and $\sigma'$ of the interchange process. We now define a coupling $((\sigma_t)_{t \ge 0}, (\sigma'_t)_{t \ge 0})$ of the interchange processes starting from these initial states. We make the edges ring at rate 2, but when an edge rings, it is ignored with probability $1/2$. We use the same clocks for both systems. If at time $t$ an edge $e=xy$ rings and $\sigma_{t-}(x)=\sigma'_{t-}(y)$ (where $\sigma_{t-}(x)\coloneqq \lim_{\delta \to 0^+  }\sigma_{t-\delta}(x)$, as usual) or  $\sigma_{t-}(y)=\sigma'_{t-}(x)$ then with probability 1/2 we set $\sigma_t=\sigma_{t-} \circ  \tau_{e}$ and $\sigma'_{t}=\sigma'_{t-}$ and with probability $1/2$ we set  $\sigma_t=\sigma_{t-} $ and $\sigma'_{t}=\sigma'_{t-} \circ \tau_e$ (either way, the number of disagreements decreases). If  $\sigma_{t-}(x) \neq \sigma'_{t-}(y)$ and  $\sigma_{t-}(y)\neq \sigma'_{t-}(x)$ then with probability 1/2 we set  $\sigma_t=\sigma_{t-} \circ  \tau_{e}$ and  $\sigma_t'=\sigma_{t-} ' \circ \tau_{e}$   and with probability 1/2 we set  $\sigma_t=\sigma_{t-}  $ and  $\sigma_t'=\sigma_{t-}  '$.
 
 We see that for all $i$ once the particle labeled $i$ is coupled in the two systems, it remains coupled. That is, if $\sigma_{t}^{-1}(i)=(\sigma'_{t})^{-1}(i)$ then for $t'>t$ we also have
 $\sigma_{t'}^{-1}(i)=(\sigma'_{t'})^{-1}(i)$. Whenever the position of particle $i$ is adjacent in one system is adjacent to the current position of particle $i$ in the other system (i.e.\ $\sigma_{t}^{-1}(i)(\sigma'_{t})^{-1}(i) \in E_n$) the infinitesimal rate in which they are coupled is $2$.

\begin{lemma}\label{l:upperbound}There exists a $C$ such that
\[
  \mix(\Cay(\sym_n,S_n(\eps,u_n)))\le \frac{Cu_n4^{u_n}}{\eps^4},
\]
under our usual assumption that $\eps >u_n^{-1/3}$. 
  \end{lemma}
\begin{proof}Lemma \ref{l:NKt} shows that a.a.s.\ indeed all particles in one system spend at least $\frac{2}{3}$ of the time in $K$ by time $C_1L/\eps^4$ for any $C_1$ sufficiently large. By a union bound this applies to both systems in the above coupling. On this event (occurring for both systems),
  for each $i$ the particle labeled $i$ has to spend at least  $1/3$ of the time by time $C_1L/\eps^4$ in $K$ simultaneously in both systems. Since the particle gets coupled with rate 2 during these times, a standard argument shows that %(which implies that the total amount of time they both spend simultaneously in the complete graph before getting coupled is dominated by the exponential distribution with parameter $2$),
  %\footnote{\tb{This would be equality in law if we change the coupling so that the particles can only couple at time $s$ such that at time $s_{-}$ they are both in $K$. It takes longer to rigorously verify the last intuitive statement, than it does to give and alternative proof for \eqref{eq:rate2bound} by reducing to discrete time. The probability that both particles spent simultaneously at least $t/3$ time units in $K$ by time $t\coloneqq  C_1L/\eps^4$ but  made less than $tn/12 $ jumps (by time $t$) when in $K$ is overwhelmingly small (i.e.\ is $o(1/n)$), and hence by a union bound can be ignored. On the complement of the last event, the chance they failed to couple is at most the chance that a $\mathrm{Bin}(\lfloor tn/12 \rfloor,\frac{1}{|K|})$ r.v.\ equals 0. The last probability is exponentially small in $t$, and hence is $o(1/n)$, provided $C_1$ is sufficiently large (recall that $L=u4^u\asymp\log n$).}}
the conditional probability of particle $i$ not getting coupled is at most
\begin{equation}
\label{eq:rate2bound}
\exp(-cC_1 L/\eps^4)\le\exp(-cC_1L)
\stackrel{(*)}{\le}n^{-cC_1},
\end{equation}
where the inequality marked by $(*)$ follows since $L=u4^u\asymp\log n$ (see just above \eqref{eq:gadgetsmall}). If $C_1$ is sufficiently large, this will be $\ll 1/n$ and we may apply a union bound and get that a.a.s.\ all particles are coupled by time $C_1L/\eps^4$. As the initial states $\sigma$ and $\sigma'$ are arbitrary, this implies that the mixing time is at most $C_1L/\eps^4$ (see e.g.\ \cite[Theorem 5.4]{cf:LPW}).
\end{proof}
\subsection{The perturbation}
\label{s:pertoverview}
In this section we analyse the perturbed versions of $S_n$, lower bound their mixing time, and thus conclude the proofs of Theorems \ref{thm:short} and \ref{thm:weighted}. The following convention will be useful here and in other places in the paper. Thus we make special note of it
\begin{definition}\label{def:left}We call an edge of $H_i$ that belongs to a path that is a stretching of a left edge (of $H_i'$) a `left edge'. Similarly for right edges.
\end{definition}
Do not be confused with the definition of $g$. It is still the case that $g$ counts left edges before stretching, not all left edges of $H_i$.
\begin{proof}[Proof of Theorem \ref{thm:weighted}]Recall that we are given a function $1\ll f(n)\le\log\log\log n$ and we need to construct generators $S_n$ and weights $W_n=(w_n(s))_{s\in S}$ satisfying $1\le w_n(s)\le 1+(f(n!)/\log\log n)^{1/4}$ such that
\[
\mix(\Cay(\mathfrak{S}_{n},S_{n},W_n)) \gtrsim \mix(\Cay(\mathfrak{S}_{n},S_{n}))f(n!).
\]
Define $\eps\coloneqq   c_1(f(n!)/\log\log n)^{1/4}$ where $c_1$ is a universal positive constant that will be fixed soon (but let us already require $c_1<\frac14$).  The requirement $\eps>u^{-1/3}$ will be satisfied for $n$ sufficiently large. %(recall that $u \asymp \log \log n$).
We use the set $S_n(u,\eps)$ defined above with $u=u_n$ and this $\eps$ (we remind that $u_n\asymp \log\log n$). %Recall that, given $n$, it depends on an additional parameter $\eps$ (the parameter $u$ is fixed by $n$, we remind that $u\asymp \log\log n$). Let us fix some $n$ and remove it from the notation, i.e.\ write $S$, $W$ and $w(s)$. %\tb{(recall that $S$ are the set of transpositions $\{\tau_e:e \in E(G)\}$; below by abuse of notation we identify $\tau_e$ with $e$ by refering to the former as a left edge whenever $e$ is a left edge)}. As explained, defining $\eps$ defines $S$, as $\eps$ is the only free parameter in the construction.
  
  %o define $W$, recall that when we constructed $S$ we took some trees $H_i'$ and stretched their edges; and that we (arbitrarily) labelled each edge in each $H_i'$ either `left' or `right'. Let $L$ be the set of all edges of $H_i$ that belongs to a stretching of a left edge of $H_i'$.
  Denote, for any $\delta>0$, $W(\delta,n)=(w(s))_{s\in S_n}$ with 
\[
  w(s)\coloneqq \begin{cases}
    1+\delta & s\textrm{ is a left edge}\\
    1&\textrm{otherwise}
  \end{cases}
\]
(`otherwise' referring to both right edges and to edges of $K$). We will take $\delta=\eps/c_1$ in what follows. The notation $W(\delta,n)$ will be reused below in the proof of Theorem \ref{thm:short}, but there we will take $\delta=3$, so let us proceed under the assumption $\delta\le 3$, which holds under the definitions of $\delta$ and $\eps$ above too.

Recall the notation $g(v)$ for the number of left children in a path from the root to $v$ (before stretching). Examine first an infinite binary tree where each left child has weight $1+\delta$ for some $\delta>0$, and each right child has weight 1 (denote this object by $\cT_\delta$). Let $Y_k$ be the \emph{last} vertex in the $k^{\textrm{th}}$ level visited by random walk on it. By \cite[Fact 4.1 (2 a)]{HP} (proved in the appendix of \cite{HP}), $g(Y_k)$ has the same distribution as the sum of $k$ independent $\{0,1\}$-variables taking the value 1 with probability
\[
\frac{\sqrt{1+\delta}}{1+\sqrt{1+\delta}}=\frac{1}{2}+\delta/4+O(\delta^2).
\]
This fact holds also for random walk on $\cT_\delta$ started from either child of the root and conditioned not to return to the root. The proof in \cite{HP} applies to this case verbatim.
Now, if $c_1$ is sufficiently small then
\[
\eta\coloneqq \frac{\sqrt{1+\eps/c_1}}{1+\sqrt{1+\eps/c_1}}>\frac12+3\eps.
\]
(recall that $\delta=\eps/c_1$ is bounded above by $3$). Fix $c_1$ to satisfy this property.

Still on the infinite tree $\cT_{\delta}$, denote by $Y_k^*$ the vertex where the walker is at on the \emph{first} time it hits level $k$, in other words, the hitting point. It is straightforward to see that $\Pr(|g(Y_k^*)-g(Y_k)|>\lambda)\le 2e^{-c\lambda}$ for every $\lambda$, where the (non-negative) constant $c$ is   independent of $\eps$.

Information on $Y_k^*$ can already be translated to our graphs $H_i$, because random walk on $H_i$, when considered only at times when it reaches a vertex before stretching, is identical to random walk on a piece of $\cT_{\eps/c_1}$ (say, by Lemma \ref{lem:conductances}). We get that a random walk starting from a root of $H_i$ and conditioned not to go to $H_{i-1}$ before leaving $H_i$ (for $i=1$, an unconditioned walker) has, when it exits $H_i$ that $g$ is distributed like $\textrm{Bin}(s_i,\eta)$ plus a quantity with a uniform exponential tail (uniform in both $i$ and the value attained by the $\textrm{Bin}(s_i,\eta)$ random variable). %\tb{[What is $k$ here? Shouldn't it say  $\textrm{Bin}(k,\eta)$ instead of  $\textrm{Bin}(\eta,k)$? By "bounded quantity" do we mean, a random quantity whose tail decays exponentially?]} %\tb{We have also used the fact that each tree in $H_i$ can be transformed into a binary tree by a network reduction. The edge weight of a left edge after such a reduction is $1+\delta$ and of a right edge is $1$.}

A similar argument shows  that, now on our graphs $H_i$, if $Z_i$ is the first vertex the walker is in among the roots of $H_{i+1}$ (which of course is also a leaf of $H_i$) and $Z_i^*$ is the last vertex the walker is in $H_i$, (say, before hitting  the leaves of $H_{i+1}$ or $K$ for the first time),
then $|g(Z_i)-g(Z_i^*)|$ is bounded with an exponential tail (uniformly in $i$). %\tb{[what do you mean bounded with an exponential tail? That the exponential tail is uniform in $i$?]}

We may now finish the proof of the theorem. Indeed, let $X$ be the particle that was at time 0 at the root of $H_1$. Let $T$ be the time $X$ hits the leaves of $H_1$. We see that, if $\lambda>0$ is some sufficiently small constant then
\begin{align*}
  \Pr(T\ge \lambda u4^u)&>1-Ce^{-cs_1}=1-Ce^{-cu}\\
  \Pr(g(X_T)>(\tfrac12 +2\eps)s_1)&>1-Ce^{-cs_1\eps^2}>1-Ce^{-cu^{1/2}}
\end{align*}
where the last inequality is due to $\eps\gtrsim(\log\log n)^{-1/4}$. 
%with high probability it hits the leaves of $H_1$ at time $\asymp u4^u$, and then has, with high probability $g>(\frac12+\eps)s_1$,
In particular, with the same probability $X_T$ is in $B_1$. The same $X$ still has that $g>(\frac12+\eps)s_1$ when leaving $H_1$ (again with probability $>1-Ce^{-cs_1\eps^2}$), and then hits the leaves of $H_2$ after another at least $\lambda u4^u$ time units, and hits $B_2$, and so on. We get that at time $\lambda u^24^u$ this particle is still inside the gadget, with probability at least $1-Cue^{-cu^{1/2}}$. This of course means the walk on $\sym_n$ is not yet mixed. Hence
\[
\mix(\Cay(\sym_n,S_n,W_n))\gtrsim u^24^u.
\]
With Lemma \ref{l:upperbound} %(which we may apply since, again, $\eps\gtrsim(\log\log n)^{-1/4}$ and $u\asymp\log\log n$)
we get that
\[
\mix(\Cay(\sym_n,S_n,W_n))\gtrsim u^24^u\gtrsim \mix(\Cay(\sym_n,S_n)) \cdot u\eps^4\asymp
\mix(\Cay(\sym_n,S_n))f(n!)
\]
as claimed (in the last `$\asymp$' we used $u \asymp \log \log n$). This concludes the proof.
\end{proof}
\begin{proof}[Proof of Theorem \ref{thm:short}]
  Recall from the proof sketch \S\ref{s:sketch} that $S_n'$ is created by adding to each path of $S_n$ that came from stretching a left edge, edges between even vertices (initially at distance two from one another). The parallel-serial laws show that the resistance of a path of length $2N$ to which such edges have been added is $\frac23 N$. Examining a walker only at times where it is in vertices that were not added in the stretching process, we see that its walk is exactly identical to a walk on $\Cay(\sym_n,S_n,W(3,n))$, where $W(3,n)$ is from the previous proof. %(we apply Lemma \ref{lem:conductances} between these times, and calculate the resistances using simple network reductions).
Hence choosing $\eps=3c_1$ we get, as in the previous proof, $\mix(\Cay(\sym_n,S_n'))\gtrsim u^24^u$.

  This almost finishes the proof of Theorem \ref{thm:short}. The only remaining issue to address is that Theorem \ref{thm:short} is formulated in discrete time, while we worked all along in continuous time. This is not a problem. Indeed, if $P$ is a transition matrix and $I$ is the identity matrix, then the total variation mixing time $\mix^{\delta \, \mathrm{lazy}}$ of the $\delta$-lazy chain with transition matrix $\delta I+(1-\delta) P$ and that of the continuous-time chain with generator $\cL=P-I$, denoted by $\mix^{\mathrm{ct}} $, satisfy 
\begin{equation}
\label{e:lazyvsctstime}
\frac{\delta}{C(1-\delta) } (\mix^{\delta \, \mathrm{lazy}}-c_{\delta})\le \mix^{\mathrm{ct}} \le \frac{C}{(1-\delta) }(\mix^{\delta \, \mathrm{lazy}}+c_{\delta})
\end{equation}
for an absolute constant $C>0$ and a constant $c_{\delta}>0$, independent of the Markov chain. The case $\delta=1/2$ follows directly from 
\cite[Theorem 20.3]{cf:LPW} and the argument extends to all $\delta \in (0,1) $.   For much finer relations between the two mixing times in the reversible setup see \cite{tave}, \cite{cf:Basu} and \cite{CSC}.

In our case, we estimated the continuous time mixing time with the rates equal to 1, while the generator $\cL$ has rates $1/|S_n|$ or $1/|S_n'|$, as the case may be. Multiplying all the rates by a constant changes the mixing time by the same constant, so we get
\begin{align*}
\mix(\Cay(\sym_n,S_n))
  &\stackrel{\mathclap{\textrm{(\ref{e:lazyvsctstime})}}}{\lesssim}
  |S_n|\mix^{\textrm{ct}}(\Cay(\sym_n,S_n))\lesssim |S_n|u4^u\asymp n^2\log n\\
\mix(\Cay(\sym_n,S_n'))
  &\stackrel{\mathclap{\textrm{(\ref{e:lazyvsctstime})}}}{\gtrsim}
  |S_n'|\mix^{\textrm{ct}}(\Cay(\sym_n,S_n,W(3,n)))\gtrsim |S_n'|u^24^u\asymp n^2\log n\log\log n.
\end{align*}
The theorem is thus proved.
%Taking $\kappa=|S_n|$ or $|S_n'|$ gives the translation from the interchange process (with all rates of all edges 1) to a discrete-time lazy walk on $\sym_n$ (recalling that $|S_n'|=|S_n|(1+o(1))$).
\end{proof}

\section{Proof of Theorem \ref{thm: 2}}
\label{s:proofthm2}
%Throughout we do not distinguish in the notation between a graph $G$ and its set of vertices, e.g.\ $v\in G$ means that $v$ is a vertex of $G$. The set of edges of $G$ will be referred to by $E(G)$.

Recall that we wish to construct a sequence of graphs $G_n$ with bounded degrees and weights with $1\le w_n(e)\le 1+o(1)$ such that the mixing time of $G_n$ is significantly smaller than the mixing time of the weighted version.

As a building block in our construction, we will need the auxiliary graph described in the following lemma, whose proof is deferred to \S\ref{s:auxexpander}.
\begin{lemma}
\label{lem:auxexpander}
There exists an absolute constant $\mu >0 $ such that for every $m$ there exists a graph $H$ of maximal degree $6$ with $|H| \asymp 2^{10m}$ containing two disjoint sets of vertices $B$ and $W$ of sizes $|B|=2^m$ and $|W|= 2^{10m}$  such that lazy simple random walk on $H$ satisfies that %\textcolor{blue}{for some $B \supset A$ such that $|B| \le 2n$ we have that:}
\begin{alignat}{2}
\label{e:uniformhitting}
\Pr_b[T_{B \setminus \{b\}} < T_{W} ]&\lesssim 2^{-4m}&\qquad&\forall b\in B\\
\label{e:uniformhitting2}
\ex_h[T_{B \setminus \{h\}}] &\asymp |H|/2^m  &&\forall h\in H  \\
\label{e:uniformhitting3}
\Pr_h[T_{B \setminus \{h\}}>|H|/2^m] &\gtrsim 1  &&\forall h\in  H\\
\label{e:atleastmu}
\Pr_b[T_W<T_b\,|\,T_W<T_{B\setminus \{b\}}] &\ge\mu &&\forall b\in B.
\end{alignat}
%Moreover, the law of $|\{t \in [0, T_{W}):X_{t} =a\}|$ given $\{X_0=a\} \cap \{T_{W}<T_{A \setminus \{a\}}\}$, is the same for all $a \in A$ and has a Geometric distribution with parameter $1-\mu_n$ for some $\mu_n$ such that $\mu_n \le \mu$.
Moreover, the last probability is the same for all  $b \in B$. % and so is the first.
Lastly,    for all  $w\in W$  starting from $w$  the hitting distribution of $B$ uniform. %the same, and satisfies $||\mathbb{P}_v(X_{T_A}=\cdot)-\mathrm{Unif}||_{\TV}\le Cn^{-4}$. %\tb{[I don't care about this much, but elsewhere we used $\|\cdot \|_{\TV}$ and not $d_{\TV}$, so a stickler referee might ask us to define  $d_{\TV}$.]}
\end{lemma}

As usual, $T_W$, $T_B$, etc.\ are the hitting times of $B$, $W$, etc.

%\begin{remark}
%\label{r:automorphism}
%\textcolor{blue}{For all $w \in W$ and $b \in B$ we have that the graph distance between $w$ and $b$ is $9m$. For $h \in H$ denote the graph distance of $h$ from $W$ by $\ell(h)$, and let $L_i:=\ell (X_{i+1})-\ell (X_{i})$. It is also not hard to verify that for all $w \in W $, starting from $w$ we have that $X_{T_{B}} \sim \text{Unif}(B)$ and is independent of $(T_B, \left(L_{i}\right)_{i=0}^{T_{B}-1})$, whose distribution is the same for all $w \in W$, and similarly that for all $b \in B $, starting from $b$ we have that $X_{T_{H}} \sim \text{Unif}(H)$ and is independent of $(T_H, \left(L_{i}\right)_{i=0}^{T_{H}-1})$, whose distribution is the same for all $b \in B$.
%}\end{remark}
\subsection{The construction}
\label{s:bdd construction}
Let $G=G_n((\log\log n)^{-1/8})$ i.e.\ the graph from the construction of Theorem \ref{thm:weighted} with the parameter $\eps$ from the construction taken to be $(\log\log n)^{-1/8}$ ($G$ is the graph on which the interchange process is performed, so $E(G)$ is a set of transpositions of $\sym_n$). Let $m$ satisfy that $2^{m-1}< |E(G)|\le 2^{m}$. Let $H$ be the graph from Lemma \ref{lem:auxexpander} with this $m$ (so $|H| \asymp 2^{10m} \asymp n^{20}$). Let $A\subseteq B$ ($B$ from the statement of Lemma \ref{lem:auxexpander}) be some arbitrary set of size $|E(G)|$ and let $\tau:A\to E(G)$ be some arbitrary bijection. %\tb{[Is this $\tau$ the same as $\tau_a$? before I was using $\tau_e$ generically as the transposition corresponding to an edge $e$. I suspect that you do not distinguish below between $e$ and the corresponding transposition... I think it is better to not use $\tau$ here as the bijection and just say that we assume that $A$ is labeled (bijectively) by the set $E(G)$.]}

We now construct our graph, which we denote by
$L$. We take the vertex set to be $H\times \sym_n$. We define the edges implicitly by describing the transition probabilities of the random walk. Let $\{a,b\}\in  E(H)$. If $a\not\in A$ we set $P\left(\left(a,\sigma \right),\left(b,\sigma\right) \right)=\frac{1}{\deg a}$ for all $\sigma \in \sym_n$, where $\deg a$ is the degree of $a$ in $H$. If $a \in A$ and $b \in   H \setminus A  $  we set $P\left(\left(a,\sigma \right),\left(b,\sigma\right) \right)=\frac{1}{2\deg a}$, while $P\left(\left(a,\sigma \right),\left( a,\sigma \circ \tau_{a} \right) \right)=\frac{1}{2}$
(recall that $\tau_a$ is the transposition corresponding to $a$). \tb{}  No other transitions have positive probability. Below we consider the mixing time of the continuous time version of $P$ or the discrete time mixing time of $\frac 12 (I+P)$. We shall denote either mixing time by $\mix(L)$.

%For ease of notation we now drop the subscript $n$. 
This chain $(X_t,\sigma_t)_{t \ge 0}$ can be described as follows: We have a random walk $(X_t)_{t \ge 0}$ on $H$ and an `interchange process' $\sigma_t$ on $G$ which evolves in slow motion. Whenever the walk $X_t$ on  $H$ is at some vertex $a \in A$, %with equal probability it either makes a SRW step on $H$ (i.e.\ $X_{t+1}$ is a random neighbour of $X_t$ and $\sigma_{t+1}=\sigma_t$) or in the `interchange process'  the particles occupying the end points of the edge corresponding to $a$ in $G$ are swapped (i.e.\ $X_{t+1}=X_t$ and $\sigma_{t+1}=\sigma_t \circ \tau_a$). Whenever $X_t \notin A$ we have that $\sigma_{t+1}=\sigma_t$ and $X_{t+1}$ is a random neighbour of $X_t$ in $H$.
it either stays put or makes a random walk step on $H$. If it stays put in $a \in A$ then it also makes one step of the interchange process, updating its state to $\sigma_t \circ \tau_a$. 
\subsection{Analysis of the example}
We first define a sequence of random times. Recall the set $W$ from the construction of $H$ in Lemma \ref{lem:auxexpander}. Let $S_1\coloneqq \inf\{t\ge 0:X_t\in W\}$ and $T_1\coloneqq \inf\{t>S_1:X_t \in B \}$. Inductively, set \[S_{i+1}\coloneqq \inf\{t>T_i:X_t \in W \} \quad \text{and} \quad T_{i+1}\coloneqq \inf\{t>S_i:X_t \in B \}.\] Let $\cJ$ be the event that for all $i \le n^4$  the walk does not visit  $B \setminus \{X_{T_i}\} $ between time $T_i$ and $S_{i+1}$. By \eqref{e:uniformhitting} $\Pr_{y}(\mathcal{J})> 1-O(n^{-4})$ for all $y\in L$. %is the same for all $y \in L$
%By Lemma \ref{lem:auxexpander} $\Pr_{y}(\mathcal{J}^c)\lesssim n^{-4}$ for all $y\in L$. %is the same for all $y \in L$ and $\eta:=\frac{\Pr_{y}(\mathcal{J}^c)}{\Pr_{y}(\mathcal{J})} \lesssim 2^{-4m}n^4=O(n^{-4})$.

Let $Z_i:=X_{T_i}$ and   $\widehat \sigma_{i}:= \sigma_{T_i}$ ($\sigma_{t}$ being the second coordinate of the chain $(X_t,\sigma_t)$). By Lemma \ref{lem:auxexpander} we have that $Z_1,\ldots,Z_{n^4}$ are i.i.d.\ uniform on $B$. 
%and that $Z_i$ and $\widehat \sigma_{i}$ are independent for all $i$.
Under $\cJ$, the behaviour of the permutation $\sigma$ in the time interval $[T_i,T_{i+1})$ is quite simple: if $Z_i\in B\setminus A$ then it does not change at all in this interval, and if $Z_i\in A$ then it is composed with $\tau_{Z_i}$ with probability $\frac 12$ for each time $t\in[T_i,T_{i+1})$ when $X_t=Z_i$. %ordinate did not change in the time interval $[S_i,T_i]$ and we get that $\sigma_{S_i}=\sigma_{T_i}=\widehat \sigma_i$). %Therefore, we may couple them to i.i.d.\ uniform samples of $A$, and the coupling succeeds with probability at least $1-C|A|^{-2}$. Let $\mathcal{J}_2$ be the event that the coupling succeeds, and let $\mathcal{J}\coloneqq \mathcal{J}_1\cup\mathcal{J}_2$. Condition on $\mathcal{J}$.
This, together with \eqref{e:atleastmu} imply that, still under $\mathcal{J}$, $(\widehat \sigma_i)_{i=1}^{n^4}$ evolves precisely like a lazy version of the discrete-time interchange process on $G$. The laziness has two sources: the probability to hit $B\setminus A$ (which gives laziness $|B\setminus A|/|B|$, which is bounded above by $\frac12$), and an additional laziness coming from the event of applying the transposition $\tau_{Z_i}$ an even number of times between $T_i$ and $S_{i+1}$. We use here the fact that the probability in \eqref{e:atleastmu} is the same for all  $a \in A$. In other words, we have a coupling of $\widehat\sigma_i$ and lazy interchange on $G$ which succeeds (i.e.\ the two processes are the same) with probability $1-O(n^{-4})$.
% in which at each step $i+1$ we pick a random edge $e \in E$ and then (independently of $e$) we set $\hat \sigma_{i+1}=\hat \sigma_i$ with probability $\frac12$ and  $\hat \sigma_{i+1}=\hat \sigma_i \circ \tau_e $ with  probability $\frac12$. %Crucially, $q$ is uniformly bounded away from $1$ (w.r.t.\ $n$; we omit the details of this and of why $q>1/2$. The fact that $q>1/2$ is not used below).

Let $r$ be the $\frac 14$ total variation mixing time of this lazy discrete-time interchange process on $G$. To estimate $r$, note that by Lemma \ref{l:upperbound}, the mixing time of the interchange process %\tb{[It sounds slightly odd to say "random walk of the interchange process" I guess this is intentional?]}
is at most $Cu4^u\eps^{-4}\asymp \log n(\log\log n)^{1/2}$ (recall that $ u4^u \asymp \log n$). Using \eqref{e:lazyvsctstime} we may translate this to the mixing time of the lazy discrete-time interchange process and get that $r\lesssim n^2 \log n(\log\log n)^{1/2}$ (recall that $|E(G)| \asymp n^2$). %%%%%and that $|B| \gg 1$)}.
In particular, $r<n^4$ for all sufficiently large $n$.

Thus, under $\cJ$ we have that $X_{T_{r+1}}$ has its first coordinate uniform on  $B$ and its second approximately uniform on $\sym_n$ and independent of the first coordinate. Removing the requirement that we are on  $\mathcal{J}$, the distribution of $X_{T_{r+1}}$ is still approximately uniform (in the TV distance) on the same set, simply because $\Pr(\mathcal{J})>1-Cn^{-4}$. %\textcolor{blue}{Indeed, fixing some $y\in \widehat V$   and writing  $\widehat Y_i:=Y_{T_i}=(Z_{i},\widehat \sigma_{i})$, by a na\"ive application of Bayes' rule, we get that for all $F \subset \widehat V$  \[\mathbb{P}_{y}[\widehat Y_{r+1} \in F\mid \mathcal{J}] \le \frac{\mathbb{P}_y[\widehat Y_{r+1} \in F]}{\Pr_{y}(\mathcal{J})} \le\mathbb{P}_y[\widehat Y_{r+1}  \in F]+\eta. \]  Denoting the uniform distribution on $B \times \sym_n $ by $\widehat \pi$ and taking $F :=\{(b,\sigma) \in B \times \sym_n :\mathbb{P}_y[\widehat Y_{r+1} =(b,\sigma) ] <\frac{1}{|B|n!} = \widehat \pi(b,\sigma)  \}$ yields that
%\begin{equation*}
%\begin{split}
%\|\mathbb{P}_y(Y_{T_{r+1}} \in \cdot )-\widehat\pi \|_{\TV}& =\widehat\pi(F)-\mathbb{P}_y(\widehat Y_{r+1} \in F) \\ & \le\widehat\pi(F)-  \mathbb{P}_y[\widehat Y_{r+1} \in F\mid \mathcal{J}]+\eta
%\\ & \le \|\mathbb{P}_y[\widehat Y_{r+1} \in \cdot \mid \mathcal{J}]-\widehat\pi \|_{\TV}+\eta\le 1/4,
%\end{split}
%\end{equation*}
%where in the last inequality we used the fact that \[\|\mathbb{P}_y[\widehat Y_{r+1} \in \cdot \mid \mathcal{J}]-\widehat\pi \|_{\TV}=\|\mathbb{P}_y[\widehat \sigma_{r+1} \in \cdot \mid \mathcal{J}]-\text{Unif}(\sym_n) \|_{\TV} \le \frac 14 - \eta, \] where the last inequality follows from  the choice of $r$ and the equality follows from the fact that given  $\mathcal{J}$ we have that $Z_{r+1}=X_{T_{r+1}}$ is uniformly distributed over $B$  and is independent of $\widehat \sigma_{r+1} =\sigma_{T_{r+1}}$.  } 

In the language of \cite{LW98}, $T_{r+1}$ is an approximate forget time. As we  recall below, by combining results from \cite{LW98} and \cite{A82},  this implies that
\begin{equation}
\label{e:mixL}
  \mix(L)\lesssim \max_{y} \ex_y (T_{r+1})
  \stackrel{\textrm{\eqref{e:uniformhitting2}}}{\asymp}
  (r+2)\frac{|H|}{2^m}\lesssim n^{20}\log n\sqrt{\log\log n}
  \end{equation}
(we have $r+2$ rather than $r+1$ in the third expression, to account for the time until the walk hits $B$ for the first time. This is also why we formulated \eqref{e:uniformhitting2} for every $h\in H$ and not just for $b\in B$).

Thus we need only describe briefly the results of \cite{LW98} and \cite{A82}. In \cite{LW98}, the authors define the mixing time differently from us (see the definition of $\cH$ in \cite[\S 2.3]{LW98}). We will adopt their notation and call this quantity $\cH$. (We will not define $\cH$ here as this would take too much space. The reader can find the definition, together with many illuminating examples, in \cite{LW98}). As for the approximate forget time, it is denoted in \cite{LW98} by $\cF_{\underline{\eps}}$ (also in \S 2.3 there). Finally, the result that $\cF_{\underline{\eps}}\asymp \cH$ is a combination of theorems 3.1 and 3.2 in \cite{LW98}.

As for \cite{A82}, it defines $\tau_1$ which is the continuous time mixing time, and $\tau_2$ which is the same as $\cH$, and \cite[Theorem 5]{A82} states that $\tau_1\asymp\tau_2$ (see also \cite{PS} where $\cH$ is denoted by $t_{\mathrm{stop}}$). Thus we get
\[
\mix(L)=\tau_1\asymp \tau_2=\cH \asymp \cF_{\underline{\eps}}\lesssim \max_{y} \ex_y (T_{r+1})
\]
which justifies the first inequality of (\ref{e:mixL}) and finishes the estimate of $\mix(L)$.

\begin{remark}An alternative proof that replaces the results of \cite{LW98} with a coupling argument is a follows. Using the specific construction of the graph $H$, the expectation of the  time required in order to  couple  the $H$ coordinate is at most $\max_{b \in B}\mathbb{E}_{b}[T_W]$ (cf.\ the coupling for lazy simple random walk on a finite $d$-ary tree in \cite[\S5.3.4]{cf:LPW}).  The above analysis allows one to then couple the $\sym_n$ coordinate with the additional amount of time required having expectation at most  $Cn^{20}\log n\sqrt{\log\log n}$.      
\end{remark}

\subsection{The perturbation}
%Recall that  $H_i'$ is the graph we stretch to get $H_i$ in \S\ref{s:proofthm3}.   Let $E(L)$ be the collection of edges of $G_n$ belonging to some path of the collection of paths  connecting (in $G_n$) some $u,v \in H_i' $ for some $i$ (such a path replaces the edge $uv \in H_i'$) such that $v$ is a left child of $u$ (in some $H_i'$).  We increase  the edge weight to $1+C \eps$ for some constant $C>0$ to be determined, for all edges of the form: $\left(a,\sigma \right)\left(a,\sigma \circ \tau_{a} \right) \in \hat E_n $ with $a \in E(L)$ (for all $\sigma \in \mathfrak{S}(V_n)$).
  Recall from the previous section the stopping times $T_i$ and $S_i$, the notation $Z_i=X_{T_i}$, $\widehat \sigma_i:=\sigma_{T_{i}}$ and the event $\cJ$.
For every $a\in A$ such that the edge that corresponds to $\tau_a$ is a left edge (recall Definition \ref{def:left}), we increase the weight of the edges $((a,\sigma),(a,\sigma \circ \tau_a))$ to $1+\theta \eps$ for some $\theta$ sufficiently large, to be fixed later. Here $\eps$ is as in \S\ref{s:bdd construction}, namely $(\log \log n)^{-1/8}$.
%The analysis below works equally for the lazy or non-lazy discrete-time chains. We note that we use below the term ``lazy" steps only for times $s$ with $X_s=a=X_{s-1}$ and $\sigma_s=\sigma_{s-1} \circ \tau_a$  for some $a \in A$ (this is somewhat an abuse of terminology, but should cause no confusion).

To analyse the effect of this perturbation fix  $i<n^4$, assume $Z_i\in A$, and denote $(a,\sigma)\coloneqq (Z_{i},\widehat \sigma_i)$. We need to examine the number of times the walker traversed the edge $((a,\sigma),(a,\sigma\circ\tau_a))$ between $T_i$ and $S_{i+1}$. %, and we condition on the event $\cJ$ from the previous section.
Denote this number by $N$. Clearly, if $N$ is even then $\widehat \sigma_{i+1}= \sigma$ and otherwise it is $\sigma \circ \tau_a$. Let $p_{\mathrm{even}}$ be the probability that $N$ is even. %Let $\widehat X_t:=X_{T_{i}+t}$.  It is not hard to verify that the law of $(\widehat X_t)_{t=0}^{S_{i+1}-T_i}$,  conditioned on $\{\widehat X_0= X_{T_{i}}=a \} \cap J$ is precisely the same as that of $(X_t)_{t=0}^{T_W}$ given that $X_0=a$, conditioned on the event $\{X_0=a\} \cap \mathcal{I}$, where $\mathcal{I}:=\{T_W<T_{B \setminus \{a\}}\}$.  This is a Doob transformed Markov chain (which is killed upon hitting $W$) with state space $\Omega:=(H \setminus B)\cup \{ a\} $ whose transition probabilities are given by \[\forall (v,u) \in(\Omega \setminus W  ) \times \Omega, \quad \mathbb{P}[X_{t+1}=u \mid \{X_t=v\} \cap \{T_W>t\} \cap \mathcal{I}]=\frac{\Pr_v[\{X_1=u\} \cap \mathcal{I} ]}{\Pr_v[\mathcal{I}]}.\] Denote the Doob transformed Markov chain by $\mathbf{Q}:=(Q_t)_{t=0}^{T_W^{Q}}$, where for $D \subset \Omega$ the hitting time of $D$ by $\mathbf{Q}$ is denoted by  $T_D^{Q}:=\inf \{t:Q_t \in D \}$.
Let $q:=\mathbb{P}[X_{T_{i}+1}=a \mid X_{T_{i}}=a]$. %=\mathbb{P}[X_1=a \mid X_0=a,\mathcal{I}]=\mathbb{P}[Q_1=a \mid Q_0=a]. \]
Let $\beta$ be the probability that after jumping away from $a$ the walk returns to $a$ before hitting $W$. % (i.e.\ that $\inf \{t>\inf \{s>T_i:X_s \neq a \}:X_t=a \}<S_{i+1}$).%   and \[\beta:=\mathbb{P}[\cR \mid X_{i}=a,\cJ]=\mathbb{P}[\widehat T_{a}<T_W \mid X_{0}=a,\mathcal{I}]=\mathbb{P}[\widehat T_{a}^{Q}<T_W^{Q} \mid Q_0=a],\] where $\widehat T_a:=\inf\{t>T_{H \setminus \{a\}}:X_t =a \}$ and  $\widehat T_a^{Q}:=\inf\{t>T_{\Omega\setminus \{a\}}^{Q}:Q_t =a \}$.
By a first step analysis
\[p_{\mathrm{even}}=q(1-p_{\mathrm{even}})+(1-q)\left(1-\beta(1-p_{\mathrm{even}}) \right). \]
Solving yields that
\[p_{\mathrm{even}}=1-\frac{q}{1+q-(1-q)\beta}. \]
Conveniently, the perturbation does not affect $\beta$, it only affects $q$, increasing it from $\frac12$ to $q_0\coloneqq (1+\theta\eps)/(2+\theta\eps)$. Hence
\[\frac{\mathrm{d}}{\mathrm{d}q}p_{\mathrm{even}}=-\frac{1-\beta}{(1+q-(1-q)\beta)^{2}}. \]
The last derivative is negative and is bounded away from 0. %for $q$ bounded away from 0 and 1.

It % and where the $O(n^{-4})$ error is from the conditioning over $\mathcal{J}$.
follows from this that the perturbation decreases the probability $p_{\mathrm{even}}$ by an additive term which is $\Theta(\theta \eps )$.
Thus we see that the effect of this perturbation on the induced random walk on $\Cay(\sym_n,S_n)$ is to increase the probability that left edges are taken by $\Theta(\theta \eps)$. Recall from the proof of Theorem \ref{thm:weighted} the notation $W(\delta,n)$ for the weights on $S_n$ which give added weight $\delta$ to the left edges. Denote
\[
\delta=\frac{q_0}{1+q_0-(1-q)\beta}\left/
\frac{1/2}{3/2-\beta/2}\right.-1=O(\theta\eps).
\]
Again this gives a coupling between $\widehat\sigma_i$ to a random walk on $\Cay(\mathfrak{S}_{n},S_{n},W(\delta,n))$ which succeeds with probabilty $\Pr(\cJ)=1-O(n^{-4})$ (the probability of $\cJ$ is not affected by the perturbation).

The only condition to apply the analysis of Theorem \ref{thm:weighted} is $\eps/c_1\le\delta\le 3$, where $c_1$ is from the proof of Theorem \ref{thm:weighted}. Taking $\theta$ sufficiently large will ensure the condition $\delta\ge \eps/c_1$ while the condition $\delta\le 3$ holds for $n$ sufficiently large. Fix $\theta$ to satisfy this requirement. Thus the analysis of the proof of Theorem \ref{thm:weighted} shows that the particle that was at the root of $H_1$ at time 0 ($H_1$ from the construction of $G$, and unrelated to the $H$ from Lemma \ref{lem:auxexpander}) is still in the gadget after $r'\coloneqq  c n^2u^24^u\asymp n^2\log n\log\log n$  steps of the induced random walk, for $c$ sufficiently small. Since the coupling between $\widehat\sigma_i$ and the interchange process succeeds with high probability, this shows the same behaviour for $\widehat\sigma_i$. %This requires $\theta$ to be sufficiently large, and we fix $\theta$ at this point to satisfy this property.
This of course means that the random walk on $L$ is not mixed. Using \eqref{e:uniformhitting3} we see that with high probability, by time $cr'|H|/2^m$ the induced walk still did not do $r'$ steps, so we get
\begin{align*}
\mix(L,\textrm{perturbed weights})&\gtrsim cr'|H|/n^2 
%\stackrel{\textrm{\eqref{e:uniformhitting2}}}{\asymp}
%\frac{|H|}{n}n^2\log n(\log \log n)^2\\
\asymp n^{20}\log n\log\log n
\end{align*}
proving Theorem \ref{thm: 2}. \qed % (recalling that $u \asymp \log \log n$ and $\varepsilon \asymp (\log \log n)^{-1/8}$).\qed

\begin{figure}
\centering{\input{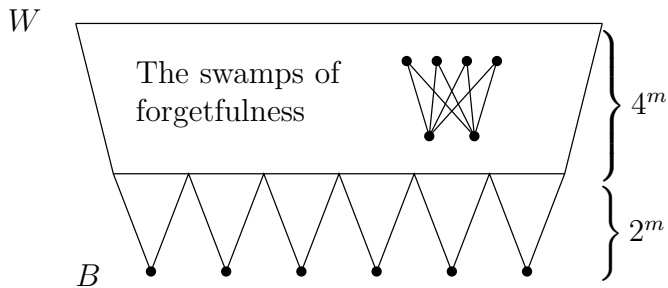}}
\caption{The `clock graph'. The triangles emanating from vertices of $B$ are the trees, the area above them is the swamp.}\label{fig:swamps}
\end{figure}

\subsection{Proof of Lemma \ref{lem:auxexpander}}
\label{s:auxexpander}
Let $s \in \N$. %For a set of size $2^s$ labeled by $[2^s]\coloneqq  \{1,\ldots,2^s \}$ we consider the following operation: Consider
For $0\le\ell\le s$ we denote $\mathcal{A}_{\ell}\coloneqq   \{u_{i_{1},\ldots,i_{\ell}}^k:i_1,\ldots,i_{\ell} \in [4],k \in [2^{s-\ell}] \}$ (for $\ell=0$ this simply means $\mathcal{A}_0=[2^s]$). %For all $j \in [2^{s-1}]$ we connect both $j$ and $j+2^{s-1}$   to the elements $u_{1}^{j},u_{2}^{j},u_{3}^{j}$ and $u_{4}^{j}$ of $\mathcal{A}_1$.
For all $\ell \le s-1$, $i_1,\ldots,i_{\ell} \in [4]$ and $k \in [2^{s-\ell-1}] $   we connect both $u_{i_{1},\ldots,i_{\ell}}^k$ and $u_{i_{1},\ldots,i_{\ell}}^{k+2^{s-\ell-1}}$ to $u_{i_{1},\ldots,i_{\ell},1}^k$, $u_{i_{1},\ldots,i_{\ell},2}^k$, $u_{i_{1},\ldots,i_{\ell},3}^k $ and $u_{i_{1},\ldots,i_{\ell},4}^k$.

  %Let $m \in \N$ be such that $2^{m-1}<n \le 2^m$.
 We start the construction of $G$ with $2^m$ binary trees of depth $4m$. The set $B$ is taken to be the collection of the $2^m$ roots. %An arbitrary set of $n$ roots of these trees will be the subset $A$. %We label the roots by $1,\ldots2^m$. For each binary tree, we apply the above operation on its leaf set, where it is labeled from $1$ to $2^{4m}$ from left to right. Now we have for each of the $2^m$ graphs a set of the form $\mathcal{A}_{4m}$ of size $2^{8m}$. That is, for each $i \in [2^m]$, in the $i$th graph (whose root is labeled by $i$) there is a set  $\mathcal{A}_{4m}(i)$ of size $2^{8m}$ corresponding to the vertices of distance $8m$ from the root $i$.   
 We label the union of the leaves of these trees by $[2^{5m}]$ so that each tree occupies an interval of values and identify it with $\mathcal{A}_0$ %apply the  above operation to this set
(with $s=5m$, of course). Denote $W\coloneqq  \mathcal{A}_{5m}$. %Take a $4$-regular graph $H$ with $|W|$ vertices whose spectral-gap is at least $1/10$ (e.g.\ one may be constructed randomly if $n$ is sufficiently large, see \cite{F08}). Identify $W$ with $H$ arbitrarily and add the edges of $H$ to our graph.
 This terminates the construction the graph from the statement of Lemma \ref{lem:auxexpander}, denoted by  $G$. The construction is depicted in Figure \ref{fig:swamps} with the trees depicted as triangles. The area above them, nicknamed `the swamps of forgetfulness', is composed of elements as in Figure \ref{fig:swamps}, namely, two vertices below, 4 vertices above all edges between them. These elements have the property that the particle forgets one bit whenever it traverses such an element, be it in the up or down direction. When the particle has traversed the swamp fully, it has completely forgotten its starting point.  This construction is borrowed from \cite[\S6.2]{swamps}.    %\tb{[In the proof of Theorem 3 we refer to the entire graph we are constructing in this lemma as $H$ (whereas above this is just something we put on top of the leaves) and $G$ is reserved to the graph from theorems 1 and 2. Not sure if this is a problem...]}

All of \eqref{e:uniformhitting}-\eqref{e:atleastmu} follow because the distance from the roots behaves like random walk on $\N$ with a drift. Equation \eqref{e:uniformhitting} follows because this requires to get to distance $4m$ from the roots and then back up. Equations \eqref{e:uniformhitting2} and \eqref{e:uniformhitting3} follow because with positive probability the walker hits $W$ and then needs to back up $9m$ levels. Equation \eqref{e:atleastmu} is the easiest of the four, given \eqref{e:uniformhitting}. %can be deduced by combining a first moment calculation $\mathrm{P}_a[T_{A \setminus \{a\}} \le t] \le \mathbb{E}_a[N(t,A \setminus \{a\})]$, where $N(t,A \setminus \{a\})\coloneqq  \{i \in [t]:X_i \in A \setminus \{a\}  \}$, with  \[\mathrm{P}_a[T_{A \setminus \{a\}} \le t]=\frac{\mathbb{E}_a[N(t,A \setminus \{a\})]}{\mathbb{E}_a[N(t,A \setminus \{a\}) \mid N(t,A \setminus \{a\}) >0]} \ge \frac{\mathbb{E}_a[N(t,A \setminus \{a\})]}{\max_{b \in A}\mathbb{E}_b[N(t,A \setminus \{a\})]},  \]   by considering $t \le c |V|/|A|$. We omit the details.
%\textcolor{red}{Lastly, for our purposes, namely, for arguing that the hitting distribution of $B$ starting from any $v \in H$ is uniform, and that the probabilities in \eqref{e:atleastmu} do not depend on $b$, it suffices to observe that for all $v \in W$ and all $x,y \in B$   there is an automorphism of $G$ which maps $W$ to $W$, $B$ to $B$, fixes $v$ and maps $x$ to $y$  (the same holds for all $v \in B$ and $x,y \in W$, but we do not need that). The following observation clarifies why such automorphisms exist and is helpful in giving an explicit description of them.  An inspection of the construction reveals that each step from $W$ towards $B$ (or in the opposite direction, but we do not need that) causes the walker to `forget one bit', and hence when it reaches $B$ it has completely forgotten its starting point and is uniform on $B$.}  \qed

Lastly, the claim that from every $v\in W$ the hitting distribution
of $B$ is uniform follows from the symmetries of the graph. Indeed,
let $\varepsilon_{1},\dotsc,\varepsilon_{5m}\in\{0,1\}$ and let $\varphi_{k}$
be the map of adding the $\varepsilon_{i}$ to the binary digits,
namely
\[
\varphi_{k}\Big(\sum_{i=0}^{\ell-1}b_{i}2^{i}\Big)=\sum_{i=0}^{\ell-1}(b_{i}+\varepsilon_{i}\mod2)2^{i}.
\]
Then it is easy to check that the map $\psi$ that takes $u_{i_{1},\dotsc,i_{\ell}}^{k}$
to $u_{i_{1},\dotsc,i_{\ell}}^{\varphi_{5m-\ell}(k)}$ is an automorphism
of $\bigcup\mathcal{A}_{\ell}$ (as a graph). If, in addition, $\varepsilon_{1}=\dotsb=\varepsilon_{4m}=0$
then this map, restricted to $\mathcal{A}_{0}$, has the property
that if $i$ and $j$ are leaves of the same binary tree, then so
are $\psi(i)$ and $\psi(j)$, and then $\psi$ may be extended to
an automorphism of the graph $H$. By appropriately choosing $\varepsilon_{4m+1},\dotsc,\varepsilon_{5m}$
one may get an automorphism $\psi$ that takes $b$ to $b'$ for any
two points of $B$. This shows the uniformity claim.

%\sout{To see that it is uniform even when conditioning on not hitting $A\setminus\{a\}$ note that the hitting distribution on $W$ does not change if we examine the walk after the last visit to any root (not in $A\setminus\{a\}$, of course).} This is equivalent to conditioning on not hitting all the roots, which restores the symmetry, so again the hitting distribution is uniform.
%
%In the opposite direction, i.e.\ from $W$ to $A$, by the time it has reached the roots it has completely forgotten its starting point in $W$, and its distribution on the roots is uniform. Thus the only non-uniformity in the distribution of $X_{T_A}$ comes from the fact that $A$ is only a subset of the roots, and this non-uniformity can be bounded above by the probability to hit a root not in $A$ and then hit a point of $A$ before returning to $W$, an event which has probability less than $Cn^{-4}$. \qed %it indeed satisfies the symmetries mentioned after \eqref{e:uniformhitting3}, and that indeed $\sup_n \mu_n<1$.

%\section*{Acknowledgements}

\end{document}